\documentclass{amsproc}%\documentclass[11pt, a4paper]{article}

\usepackage{amsmath,amssymb}

\usepackage{amsfonts}
\usepackage{tipa}

\usepackage{CJK, CJKnumb}
\usepackage{color}      % 支持彩色
\usepackage{indentfirst}        %首行缩进宏包
\usepackage{latexsym, bm}        % 处理数学公式中和黑斜体的宏包

\usepackage{graphicx}
\usepackage{cases}
\usepackage[all]{xypic}
\usepackage{fancyhdr}
\usepackage{pifont}

\newtheorem{theorem}{Theorem}[section]
\newtheorem{lemm}[theorem]{Lemma}
\newtheorem{prop}[theorem]{Proposition}
\newtheorem{coro}[theorem]{Corollary}
\newtheorem{conj}[theorem]{Conjecture}

\theoremstyle{definition}
\newtheorem{defi}[theorem]{Definition}

\theoremstyle{remark}
\newtheorem{remark}[theorem]{Remark}
\numberwithin{equation}{section}

\def\b{\underline}

\def\dim{\hbox{dim}}

\newfont{\df}{eufm10}

\def\dim{\hbox{\rm dim}\,}

\def\ker{\hbox{Ker}\,}
\def\span{\hbox{\rm span}\,}

\def\im{\hbox{\rm Im}\,}

\allowdisplaybreaks[4]  % \eqnarray如果很长, 影响分栏、换行和分页
\CJKtilde   %用于解决英文字母和汉字的间距问题。例如：变量~$x$~的值。

\numberwithin{equation}{section}
\begin{document}

\title[Loewy filtration and quantum de Rham cohomology]
{Loewy Filtration and Quantum de Rham Cohomology\\ over Quantum
Divided Power Algebra}

\author[Gu]{Haixia Gu}
\address{Department of Mathematics, East China Normal University,
Minhang Campus, Dong Chuan Road 500, Shanghai 200241, PR
China}\email{alla0824@126.com}

\author[Hu]{Naihong Hu$^\star$}
\address{Department of Mathematics, East China Normal University,
Min Hang Campus, Dong Chuan Road 500, Shanghai 200241, PR China}
\email{nhhu@math.ecnu.edu.cn}
\thanks{$^\star$N.H., Corresponding Author,
supported in part by the NNSF (Grant 11271131) and the FUDP from the
MOE of China}

%    General info
\date{April 3, 2012}


\begin{abstract}
%As a continuation of \cite{HU},
The paper explores the indecomposable submodule structures
of quantum divided power algebra $\mathcal{A}_q(n)$ defined in \cite{HU} and its truncated objects $\mathcal{A}_q(n,
\bold m)$. An ``intertwinedly-lifting"
method is established to prove the indecomposability of a module when its socle is non-simple.
The Loewy filtrations are described for all homogeneous subspaces
$\mathcal{A}^{(s)}_q(n)$ or $\mathcal{A}_q^{(s)}(n, \bold m)$, the Loewy layers and dimensions are determined. The rigidity of these indecomposable modules is proved. An interesting combinatorial identity is derived from our realization model for a class of indecomposable $\mathfrak{u}_q(\mathfrak{sl}_n)$-modules.
Meanwhile, the quantum Grassmann algebra
$\Omega_q(n)$ over $\mathcal{A}_q(n)$ is constructed, together with the quantum de Rham
complex $(\Omega_q(n), d^\bullet)$ via defining the appropriate
$q$-differentials, and its subcomplex $(\Omega_q(n,\bold m),
d^\bullet)$. For the latter, the corresponding quantum
de Rham cohomology modules are decomposed into the direct sum of some sign-trivial
$\mathfrak{u}_q(\mathfrak{sl}_n)$-modules.
\end{abstract}

\maketitle
\section{Introduction}

\noindent{\it 1.1.} For the generic parameter $q\in\mathbb C^*$, it
is well-known that the finite dimensional representation theory of
quantum groups $U_q(\mathfrak g)$ is essentially the same as that of
the complex semisimple Lie algebras $\mathfrak g$ (see the
independent work in 1988 by Lusztig \cite{Lu1} and Rosso \cite{Ro}).
The representation theory of quantum groups $U_q(\mathfrak g)$ at
roots of unity was established in the early 90s by many authors (see
Anderson-Polo-Wen \cite{APW}, DeConcini-Kac-Procesi \cite{CK},
\cite{CKP}, \cite{CKP2}, \cite{CDP}, Lusztig \cite{Lu2}, \cite{Lu3},
\cite{Lu4},  Andersen-Janzten-Soergel \cite{AJS}, etc.). In recent years, another exciting progress has been made
towards geometric representation theory (eg. \cite{BK1}, \cite{BK2},
\cite{BG}, \cite{C}, \cite{CCC}, \cite{CPRR}, \cite{F}, etc.). The picture
looks much close to the modular case (see \cite{AJS}, \cite{CK}, \cite{CDP},
\cite{Lu2}, \cite{Lu3}, \cite{Lu4}, \cite{BG}, \cite{BMR}, \cite{MR}
and references therein). %In representation theory, this is much more
%attractive and challenging %in some sense
%than the generic case.
%\medskip
%\noindent{\it 1.2.} 
Even for the restricted quantum group
$\mathfrak{u}_q(\mathfrak{sl}_2)$, there has been drawing more
attention to the category of finite-dimensional modules since the early 90s up to now, for instance, the work of
DeConcini-Kac \cite{CK}, 
Chari-Premet \cite{CP}, Suter \cite{Su},
Xiao \cite{XJ}, and recently, Kondo-Saito \cite{KS}, etc. Their main
problems focus on determining all simple modules of
$\mathfrak{u}_q(\mathfrak{sl}_2)$; classifying and constructing the
restricted indecomposable modules of
$\mathfrak{u}_q(\mathfrak{sl}_2)$; decomposing
$\mathfrak{u}_q(\mathfrak{sl}_2)$ as principal indecomposable
modules (PIMs) and decomposing the tensor product of a PIM and a
module as a direct sum of PIMs; determining all finite dimensional
indecomposable representations of $\mathfrak{u}_q(\mathfrak{sl}_2)$;
exploring the tensor product decomposition rules for all
indecomposable modules of $\mathfrak{u}_q(\mathfrak{sl}_2)$ with $q$
being $2p$-th root of unity ($p\ge 2$), respectively, etc.

\medskip
\noindent{\it 1.2.} %While for $U_q(\mathfrak g)$
%at $q$ an $\ell$-th root of $1$, there are many work appeared. For
%example, DeConcini-Kac-Procesi \cite{CKP} proved a conjecture of
%DeConcini-Kac-Procesi on simple module's dimension divisible by a
%maximal $\ell$-power in case simple modules are of maximal dimension
%$\ell^N$ ($N=|\,\Phi^+\,|$). Cantarini \cite{C}
%constructed all the subregular simple modules (as induced modules)
%of dimensions divisible by $\ell^{N-1}$; \cite{CCC} further
%confirmed the conjecture of DeConcini-Kac-Procesi on dimension
%$\ell$-divisibility in case the simple module in question is
%parametrized by a spherical conjugate class for any semisimple Lie algebra.
In the representation theory of quantum groups at roots of unity, it is often assumed that the
parameter $q$ is a primitive $\ell$-th root of unity with $\ell$ an odd prime. Recently, there has
been increasing interest in the cases where $\ell$ is an even integer. For example, in the study of knot invariants (\cite{MN}), or in logarithmic conformal field theories where Feigin et al. (\cite{FGST, GSTF}) make a new correspondence between logarithmic conformal field theories based on the so-called triplet VOA $W(p)$ and representation theory of the restricted quantum enveloping algebras. More precisely,
they gave the following 
\begin{conj} $($\cite{GSTF}$)$ Let $p\ge 2$, ${\mathfrak u}_q(\mathfrak{sl}_2)$ be the restricted quantum enveloping algebra at $2p$-th roots of unity. As a braided quasitensor category, $W(p)$-{\bf mod} is equivalent
to ${\mathfrak u}_q(\mathfrak{sl}_2)$-{\bf mod}. Here ${\mathfrak u}_q(\mathfrak{sl}_2)$-{\bf mod} denotes the category of finite-dimensional ${\mathfrak u}_q(\mathfrak{sl}_2)$-modules.
\end{conj}

They also proved the conjecture for $p=2$. After the above conjecture, Tsuchiya and Nagatomo
proved 
\begin{theorem} 
$($\cite{NT}$)$  As abelian categories, these are equivalent for any $p\ge 2$.
\end{theorem}

These work motivated the investigations of the ``quantum group-side" of the FGST's correspondence, in particular, as tensor categories, see Kondo-Saito (\cite{KS}) and Semikhatov (\cite{S}). Note that ${\mathfrak u}_q(\mathfrak{sl}_2)$-{\bf mod} has a structure of a rigid tensor category, but it is not a braided tensor category if $p\ge 3$ (since ${\mathfrak u}_q(\mathfrak{sl}_2)$ has no universal $R$-matrices for $p\ge 3$). Kondo-Saito's main result is to determine indecomposable decomposition of all tensor products of indecomposable ${\mathfrak u}_q(\mathfrak{sl}_2)$-modules in explicit formulas. These also suggest that Conjecture 1.1 needs to be modified; although $W(p)$-{\bf mod} and ${\mathfrak u}_q(\mathfrak{sl}_2)$-{\bf mod} are equivalent as abelian categories by Theorem 1.2, their natural tensor structures do not agree with each other. 

On the other hand, Hu \cite{HU} first defined the quantum divided power
algebras $\mathcal{A}_q(n)$ and the restricted quantum divided power
subalgebras $\mathcal{A}_q(n, \bold 1)$ as
$\mathfrak{u}_q(\mathfrak{sl}_n)$-module algebras by defining the
appropriate $q$-derivations, and thereby provided a realization
model for some simple modules with highest weights
$(\ell{-}1{-}s_i)\lambda_{i-1}+s_i\lambda_i$ $(0\le s_i<\ell$).
Recently, Semikhatov \cite{S} also exploited the divided-power quantum
plane $\underline{\mathbb C}_q$ that is the rank $2$ quantum
divided power algebra $\mathcal{A}_q(2)$ and its
$\mathfrak{u}_q(\mathfrak{sl}_2)$-module algebra realization to
derive an explicit description of the indecomposable decompositions
of $(\underline{\mathbb C}_q)_{(np-1)}$ and of the space of
$1$-forms $(\Omega_q^1)_{(np-1)}$ for the Wess-Zumino de Rham
complex on $\underline{\mathbb C}_q$ (at $q$ a $2p$-th root of
$1$). 

Anyway, up to now, the study for the tensor category ${\mathfrak u}_q(\mathfrak{sl}_2)$-{\bf mod} is sufficient enough and perfect.
A natural question is to ask what about the tensor category ${\mathfrak u}_q(\mathfrak{sl}_n)$-{\bf mod}, for $n>2$.

\medskip
\noindent{\it 1.3.} In contrast to the generic case, the category ${\mathfrak u}_q(\mathfrak{sl}_n)$-{\bf mod} of
finite dimensional $\mathfrak{u}_q(\mathfrak{sl}_n)$-modules is
non-semisimple. So in this case it is necessary to pay more
attention to studying indecomposable modules. While, category ${\mathfrak u}_q(\mathfrak{sl}_n)$-{\bf mod} for $n>2$ is more complicated than ${\mathfrak u}_q(\mathfrak{sl}_2)$-{\bf mod}, as witnessed by a Theorem of Feldvoss-Witherspoon (\cite{FW}) stating that small quantum groups of rank at least two are wild, which was a conjecture of Cibils (\cite{C}), meanwhile, ${\mathfrak u}_q(\mathfrak{sl}_2)$ is known to be tame (see \cite{Su, XJ}).
In this paper, we will
focus on the restricted quantum groups
$\mathfrak{u}_q(\mathfrak{sl}_n)$ for $n>2$ and explore the indecomposable submodule
structures for $\mathcal{A}_q(n)$ and its truncated objects
$\mathcal{A}_q(n, \bold m)$ by the method of filtrations analysis, among which
Propositions 3.3---3.6 and Lemma 3.7 serve as
the basic but essential observations for the whole story.
Furthermore, we define the quantum Grassmann algebra
$\Omega_q(n)$ over $\mathcal A_q(n)$ and construct the quantum de Rham complex
$(\Omega_q(n), d^\bullet)$ via defining the appropriate
$q$-differentials $d^\bullet$  and its subcomplex $(\Omega_q(n,\bold
m), d^\bullet)$, describe the corresponding quantum de Rham
cohomology modules $H^\bullet(\Omega_q)$ for
$\Omega_q=\Omega_q(n)$ or $\Omega_q(n,\bold m)$, as well as compute the dimensions of $H^\bullet(\Omega_q)$.

\medskip
\noindent{\it 1.4.} The paper is organized as follows. Section 2
collects some notation and the results on the quantum divided power
algebra as $\mathfrak{u}_q(\mathfrak{sl}_n)$-module algebra from
\cite{HU}. In Section 3, an important notion, named ``energy degree"
is introduced, which is crucial for the description of Loewy
filtrations as well as Loewy layers of the $s$-th homogeneous
subspaces $\mathcal{A}^{(s)}_q(n, \bold m)$ (see Theorem 3.10). We develop a new ``intertwinedly-lifting"
method to prove the
indecomposability of $\mathcal{A}^{(s)}_q(n, \bold m)$ in the case when its socle is non-simple (see the proof of Theorem 3.8 (5) (ii)), and its rigidity (see Theorem 3.12)
under the assumption that $n\geq3$ and $\mathbf{char}(q)=l\geq3$. Thereby, we see that
all $\mathcal{A}^{(s)}_q(n)$'s are indecomposable and rigid (see Corollary 3.13),
and the indecomposable decomposition of $\mathcal{A}_q(n)$ is
$\mathcal{A}_q(n)=\bigoplus_{s=0}^{+\infty}\mathcal{A}^{(s)}_q(n)$.
As a by-product, since for different $s$, $\mathcal{A}^{(s)}_q(n)$'s are
not isomorphic to each other, %we see that
$\mathfrak{u}_q(\mathfrak{sl}_n)$ $(\;n\ge 3\;)$ is of infinite representation
type (cf. \cite{ASS}). Section 4 is devoted to defining the $q$-differentials by using the $q$-derivations in \cite{HU},
which are not the ``differential calculus" in the sense of Woronowicz (\cite{W}), as well as constructing the quantum de Rham complex
$\Omega_q(n)$ over $\mathcal{A}_q(n)$ (see Propositions 4.2 \& 4.4), which is different from the
Wess-Zumino de Rham complex used in \cite{Ma}, \cite{S} in the rank $1$ case. For the quantum de Rham subcomplex
$\Omega_q(n,\bold m)$, we give an interesting description of
the corresponding quantum de Rham cohomologies (see Theorems 4.6 \& 4.8).

\section{Some notation and earlier results}% from \cite{HU}}

\noindent {\it 2.1. Arithmetic properties of $q$-binomials.} Let
$\mathbb{Z}[v, v^{-1}]$ be the Laurant polynomial ring in variable
$v$. For any integer $n\geq0$, define
$$
[n]_v =\frac{v^n-v^{-n}}{v-v^{-1}}, \qquad
[n]_v!=[n]_v[n-1]_v\cdots[1]_v.
$$
Obviously, $[n]_v$, $[n]_v!\in\mathbb{Z}[v, v^{-1}]$.

For integers $m$, $r\geq0$, we have (\cite{Lu4}),
$$
\left[{m\atop
r}\right]_v=\prod^r_{i=1}\frac{v^{m-i+1}-v^{-m+i-1}}{v^i-v^{-i}}\in\mathbb{Z}[v,
v^{-1}].
$$
Thus,

\smallskip
(1) For $0\le r\le m$, $[{m\atop r}]_v=[m]_v!/([r]_v![m-r]_v!)$;

\smallskip
(2) For $0\le m<r$,  $[{m\atop r}]_v=0$;

\smallskip
(3) For $m<0$, $[{m\atop r}]_v=(-1)^r[{{-m+r-1}\atop{r}}]_v$;

\smallskip
(4) Set $[{m\atop r}]_v=0$, when $r<0$.

\smallskip
Assume $k$ is an algebraically closed field of characteristic zero
and $q\in k^*$. We briefly set
$$[n]:=[n]_{v=q}, \qquad [n]!:=[n]_{v=q}!, \qquad
\left[{n\atop r}\right]:=\left[{m\atop r}\right]_{v=q},$$ when $v$
is specialized to $q$, where $q$-binomials satisfy
$$\left[{n\atop
r}\right]=q^{r-n}\left[{{n{-}1}\atop
{r{-}1}}\right]+q^r\left[{{n{-}1}\atop r}\right].$$

Define the $characteristic$ of $q$ as in \cite{HU},
$\mathbf{char}(q):=\text{min}\{\,\ell \mid [\,\ell\,]=0,
\ell\in\mathbb{Z}_{\geq0}\}$. $\mathbf{char}(q)=0$ if and only if
$q$ is generic. If $\mathbf{char}(q)=\ell>0$ and $q\neq\pm1$, then
either

(1) $q$ is the $2\ell$-th primitive root of unity; or

(2) $\ell$ is odd and $q$ is the $\ell$-th primitive root of unity.

\smallskip
Assume that the $\mathbb{Z}[v, v^{-1}]$-algebra $\mathcal{R}$ with
$\phi : \mathbb{Z}[v, v^{-1}]\longrightarrow \mathcal{R}$ and
$\mathbf{v}=\phi(v)$, is an integral domain satisfying
$\mathbf{v}^{2\ell}=1$ and $\mathbf{v}^{2t}\neq 1$ for all
$0<t<\ell$.

\begin{lemm}\textsc{(\cite{Lu4}, chapter 34)}
$(1)$ If $t\geq 1$ is not divided by $\ell$, and $a\in\mathbb{Z}$ is
divided by $\ell$, then $\phi(\left[a\atop t\right])=0$.

$(2)$ If $a_1\in\mathbb{Z}$ and $t_1\in \mathbb{N}$, then
$\phi(\left[\ell a_1\atop \ell t_1\right])=\mathbf{v}^{\ell
^2(a_1+1)t_1}\left(a_1\atop t_1\right).$

$(3)$ Let $a\in\mathbb{Z}$ and $t\in \mathbb{N}$, write $a=a_0+\ell
a_1$ with $a_0, a_1\in\mathbb{Z}$ such that $0\leq a_0\leq \ell -1$
and $t=t_0+\ell t_1$ with $t_0, t_1\in\mathbb{N}$ such that $0\leq
t_0\leq \ell -1$, then $\phi(\left[a\atop
t\right])=\mathbf{v}^{(a_0t_1-a_1t_0)\ell
+(a_1+1)t_1\ell^2}\phi(\left[a_0\atop t_0\right])\left(a_1\atop
t_1\right)$.

$(4)$ $\phi(\left[a\atop t\right])=\mathbf{v}^{(a_0t_1-a_1t_0)\ell
}\phi(\left[a_0\atop t_0\right])\phi(\left[\ell a_1\atop \ell
t_1\right])$.

$(5)$ $\mathbf{v}^{\ell^2+\ell }=(-1)^{\ell +1}$.
\end{lemm}

According to this proposition, it is easy to get the following.

\begin{lemm}\textsc{(\cite{Lu2}, \cite{Lu4}; \cite{HU}, 1.5)}
Assume that $q\in k^*$, $\mathbf{char}(q)=\ell \geq3$.

$(1)$ Let $m=m_0+m_1\ell$, $r=r_0+r_1\ell $ with $0\leq m_0$, $r_0<\ell$,
$m_1$, $r_1\geq0$, and $m\geq r$. Then  $[{m\atop r}]=[{m_0\atop
r_0}]\binom{m_1}{r_1}$ when $\ell $ is odd and $q$ is the $\ell$-th
primitive root of unity; $[{m\atop r}]=(-1)^{(m_1+1)r_1\ell
+m_0r_1-r_0m_1}[{m_0\atop r_0}]\binom{m_1}{r_1}$ when $q$ is the
$2\ell$-th primitive root of unity, where $\binom{m_1}{r_1}$ is an
ordinary binomial coefficient.

$(2)$ Let $m=m_0+m_1\ell, 0\leq m_0<\ell, m_1\in \mathbb{Z}$, if $\ell$
is odd and $q$ is an $\ell$-th primitive root of unity,
 then $[{m\atop
\ell}]=m_1$; if $\ell$ the $2\ell$-th primitive root of unity, then
$[{m\atop \ell}]=(-1)^{(m_1+1)\ell +m_0}m_1$.

$(3)$ If $m=m_0+m_1\ell$, $m'=m'_0+m'_1\ell \in\mathbb{Z}$ with $0\leq
m_0$, $m'_0<\ell$ satisfy $q^m=q^{m'}$, $[{m\atop \ell}]=[{m'\atop \ell
}]$, then $m=m'$.
\end{lemm}

\noindent{\it 2.2. Quantum $($restricted\,$)$ divided power
algebras.} Following  \cite{HU}, 2.1, for any $\alpha=(\alpha_1,
\ldots, \alpha_n), \beta=(\beta_1, \ldots, \beta_n)\in\mathbb{Z}^n$,
define the map
$\ast\colon\mathbb{Z}^n\times\mathbb{Z}^n\longrightarrow\mathbb{Z}$
as $\alpha\ast\beta=\sum_{j=1}^{n-1}\sum_{i>j}\alpha_i\beta_j$ and a
bicharacter
$\theta\colon\mathbb{Z}^n\times\mathbb{Z}^n\longrightarrow k$ of the
additive group $\mathbb{Z}^n$ as $\theta(\alpha,
\beta)=q^{\alpha\ast\beta-\beta\ast\alpha}$. Denote
$\varepsilon_i=(0,\cdots,\underset{i}{1},0,\cdots,0)$.

The second author introduced in \cite{HU} a {\it quantum divided
power algebra} $\mathcal{A}_q(n)$ as follows. Define
$\mathcal{A}_q(n):=\span_k\{
\,x^{(\alpha)}\mid\alpha\in\mathbb{Z}^n_+ \,\}$, with $x^{(0)}=1$,
$x^{(\varepsilon_i)}=x_i$ and
$$
x^{(\alpha)}x^{(\beta)}=q^{\alpha*\beta}\left[\alpha+\beta \atop
\alpha\right]x^{(\alpha+\beta)}=\theta(\alpha,
\beta)\,x^{(\beta)}x^{(\alpha)},
$$
where
$[{{\alpha+\beta}\atop{\alpha}}]:=\prod^n_{i=1}[{{\alpha_i+\beta_i}\atop{\alpha_i}}]$,
$[{{\alpha_i+\beta_i}\atop{\alpha_i}}]=[\alpha_i+\beta_i]!/[\alpha_i]![\beta_i]!,
\alpha_i$, $\beta_i\in\mathbb{Z}_+$.

When $\mathbf{char}(q)=\ell\geq3$, denote $\tau=(\ell {-}1, \ldots,
\ell {-}1)\in\mathbb{Z}_+^n$. Set
$$
\mathcal{A}_q(n, \textbf{1}):=\span_k\left.\Bigl\{
x^{(\alpha)}\in\mathcal{A}_q(n)\,\right|\,\alpha\in\mathbb{Z}^n_+,\
\alpha\leq\tau\,\Bigr\},
$$
where $\alpha\le \tau$ $\Longleftrightarrow$ $\alpha_i\le \tau_i$
for each $i$. Obviously, this is a subalgebra of $\mathcal{A}_q(n)$
with dimension $\ell^n$, which is called the {\it quantum restricted
divided power algebra}.

\begin{lemm}
\textsc{(\cite{HU}, 2.4)} Assume
$\mathbf{char}(q)=\ell\,(\ge 3)$, then the % quantum divided power
algebra $\mathcal{A}_q(n)$ is generated by $x_i$, $x_i^{(\ell)}
\,(1\leq i\leq n)$. When $\ell$ is odd and $q$ is an $\ell$-th
primitive root of $1$, $x_i^{(\ell)} \,(1\leq i\leq n)$ are central
in $\mathcal{A}_q(n)$, and $\mathcal{A}_q(n)\cong\mathcal{A}_q(n,
\bold 1)\bigotimes_kk[x_1^{(\ell )}, \ldots, x_n^{(\ell)}]$, as
algebras.
\end{lemm}

Define an automorphism of $\mathcal{A}_q(n)$ as
$\sigma_i(x^{(\beta)})=q^{\beta_i}x^{(\beta)}$. Obviously,
$\sigma_i\sigma_j=\sigma_j\sigma_i$. In particular,
$\sigma_i=\text{id}$ for $q=1$. Define a {\it $q$-derivative} of
$\mathcal{A}_q(n)$ as $\frac{\partial_q}{\partial
x_i}(x^{(\beta)})=q^{-\varepsilon_i\ast\beta}x^{(\beta-\varepsilon_i)}$.
Briefly, denote it by $\partial_i$. Then one has
$\partial_i\partial_j=\theta(\varepsilon_i,\varepsilon_j)\partial_j\partial_i$.

The $U_q(\mathfrak{sl}_n)$-module algebra structure of
$\mathcal{A}_q(n)$ can be realized by virtue of the generators $\sigma_i^{\pm1}, \Theta(\pm\varepsilon_i), x_i,
 \partial_i$ in the quantum Weyl algebra $\mathcal
W_q(2n)$ defined by \cite{HU}.

\begin{prop}\textsc{(\cite{HU}, 4.1)}
For any monomial $x^{(\beta)}\in \mathcal{A}_q(n)$, set
\begin{eqnarray}
e_i.\, x^{(\beta)}&=&(x_i\partial_{i+1}\sigma_i)(x^{(\beta)})=[\beta_i{+}1]\,x^{(\beta+\varepsilon_i-\varepsilon_{i+1})}, \\
f_i.\, x^{(\beta)}&=&(\sigma^{-1}x_{i+1}\partial_i)(x^{(\beta)})=[\beta_{i+1}{+}1]\,x^{(\beta-\varepsilon_i+\varepsilon_{i+1})}, \\
K_i.\, x^{(\beta)}&=&(\sigma_i\sigma_{i+1}^{-1})(x^{(\beta)})=q^{\beta_i-\beta_{i+1}}x^{(\beta)}, \\
K_i^{-1}.\,
x^{(\beta)}&=&(\sigma_i^{-1}\sigma_{i+1})(x^{(\beta)})=q^{\beta_{i+1}-\beta_{i}}x^{(\beta)},
\end{eqnarray}
where $e_i,\, f_i,\, K_i,\, K_i^{-1} \,(1\leq i\leq n{-}1)$
\footnote{In this paper, the coalgebra structure of
$U_q(\mathfrak{sl}_n)$ is defined over the generators as follows:
$\Delta(K^{\pm1}_i)=K^{\pm1}_i\otimes K^{\pm1}_i,\,
\Delta(e_i)=e_i\otimes K_i+1\otimes  e_i,\,
\Delta(f_i)=f_i\otimes1+K_i^{-1}\otimes f_i,\, \epsilon(K^\pm_i)=1,
\epsilon(e_i)=0,\, \epsilon(f_i)=0,\, S(K^{\pm1}_i)=K^{\mp1}_i,\,
S(e_i)=-e_iK^{-1}_i,\, S(f_i)=-K_if_i$, for $i=1, \cdots, n-1$.} are
the generators of $U_q(\mathfrak{sl}_n)$.

This equips $\mathcal{A}_q(n)$ with a $U$-module algebra, where
$U=U_q(\mathfrak{sl}_n)$, or $\mathfrak u_q(\mathfrak{sl}_n)$
$:=U_q(\mathfrak{sl}_n)/(e_i^\ell,\, f_i^\ell,\,
K_i^{2\ell}-1, \ \forall\,\, i<n)$ at roots of $1$.
\end{prop}

Denote by $|\,\alpha\,|:=\sum\alpha_i$ the degree of
$x^{(\alpha)}\in \mathcal A_q(n)$. Set $\mathcal A_q:=\mathcal
A_q(n)$ or $\mathcal A_q(n,\bold 1)$, let
$\mathcal{A}^{(s)}_q:=\span_k\bigl\{\, x^{(\alpha)}\in\mathcal
A_q\,\big|\,\, |\,\alpha\,|=s\,\bigr\}$ be the subspace of $\mathcal
A_q$ spanned by homogeneous elements of degree $s$.

\begin{theorem}\textsc{(\cite{HU}, 4.2)}
$\mathcal{A}^{(s)}_q$ is a $U$-submodule of $\mathcal{A}_q$.

$(1)$ If $\mathbf{char}(q)=0$, $\mathcal{A}^{(s)}_q(n)\cong
V(s\lambda_1)$ is a simple module generated by highest weight vector
$x^{(\bold s)}$, where $\bold s=(s, 0, \cdots,
0)=s\varepsilon_1=s\lambda_1$, $\lambda_1$ is the first fundamental
weight of $\mathfrak{sl}_n$.

$(2)$ If $\mathbf{char}(q)=\ell \geq 3$, $\mathcal{A}^{(s)}_q(n,
\bold 1)\cong
V\bigl((l{-}1{-}s_i)\lambda_{i{-}1}{+}s_i\lambda_i\bigr)$ is a
simple module generated by highest weight vector $x^{(\bold s)}$,
where
 $s=(i{-}1)(\ell
{-}1)+s_i$, $0\leq s_i<\ell$ for $1\leq i\leq n$, and $\bold s=(\ell
{-}1,\ldots,\ell
{-}1,s_i,0,\ldots,0)=(l{-}1{-}s_i)\lambda_{i-1}+s_i\lambda_i$,
$\lambda_i=\varepsilon_1+\cdots+\varepsilon_i$ $(i<n)$ is the $i$-th
fundamental weight of $\mathfrak{sl}_n$.
\end{theorem}

Set $P_{a,b}(t):=(1+t+t^2+\cdots+t^{b-1})^a$, for $a, \,b\in
\mathbb{Z}_{\geq 0}$.

\begin{coro}
 $\dim\mathcal{A}^{(s)}_q(n, \bold
1)=$ the coefficient of $t^s$ of polynomial
$P_{n,\ell}(t)=\sum_{i=0}^{\lfloor
\frac{s}{\ell}\rfloor}(-1)^i\binom{n}{i}\binom{n{+}s{-}i\ell{-}1}{n{-}1}$.
\end{coro}

%\begin{remark}
%It seems hard to give an explicit formula for
%$\dim\mathcal{A}^{(s)}_q(n, \bold 1)$. In view of the
%DeConcini-Kac-Procesi's conjecture, one may expect that
%$\dim\mathcal{A}^{(s)}_q(n, \bold 1)$ is divisible by a certain
%$\ell$-power.
%\end{remark}

Since $\dim\mathcal{A}^{(s)}_q<\infty$ for all $s\geq 0$, they are
both noetherian and artinian modules. Thus they satisfy the
conditions of the Krull-Schmidt theorem.

\begin{lemm}\textsc{(Krull-Schmidt theorem)}
Let $M$ be a module that is both noetherian and artinian, and let
$M=M_1\oplus\cdots\oplus M_n=N_1\oplus\cdots\oplus N_m$, where $M_i$
and $N_j$ are indecomposable. Then $m=n$ and there exists a
permutation $i\mapsto i'$ such that $M_i\cong N_{i'}$, $1\leq i\leq
n$.
\end{lemm}

\noindent {\it 2.3. Quantum exterior algebra.} Recall the Manin's
quantum exterior algebra $k[A_q^{0|n}]$
$:=k\{x_1,\cdots,x_n\}/(x_i^2, x_jx_i+q^{-1}x_ix_j, i<j)$, which is
a $U$-module algebra with $e_i.\,x_j=\delta_{i+1, j}x_i$,
$f_i.\,x_j=\delta_{ij}x_{i+1}$,
$K_i.\,x_j=q^{(\varepsilon_i-\varepsilon_{i+1}\mid
\varepsilon_j)}x_j$, for $U=U_q(\mathfrak{sl}_n)$ or $\mathfrak
u_q(\mathfrak{sl}_n)$.

% denote $k[A_q^{0|n}]$ to be $\wedge_q(n)$. Denote by
%$\wedge_q(n)_{(s)}$ the $s$-th homogeneous subspace, i.e.,
%$\wedge_q(n)_{(s)}：=\left\langle\, x_{i_1}\cdots x_{i_s}\mid 1\leq
%i_1<\cdots<i_s\leq n\,\right\rangle$, then
%$\wedge_q(n)=\bigoplus^n_{s=0}\wedge_q(n)_{(s)}$.
The known fact below is independent of $\mathbf{char}(q)$.

\begin{lemm} $k[A_q^{0|n}]=\bigoplus_{s=0}^nk[A_q^{0|n}]_{(s)}$ as $U$-modules, and
$$
k[A_q^{0|n}]_{(s)}=\span_k\bigl\{\, x_{i_1}\cdots x_{i_s}\,\big|\,
1\le i_1\le\cdots\le i_s\le n\,\bigr\}\cong V(\lambda_s)
$$
is a simple module generated by highest weight vector $x_1\cdots
x_s$, where $\lambda_s$ is the $s$-th fundamental weight of
$\mathfrak{sl}_n$.
\end{lemm}

\noindent {\it 2.4. Convention.} In the rest of paper, we will focus
our discussions on the case when $\mathbb{Q}(q)\subseteq k$,
$\mathbf{char}(q)=\ell ~(\ge 3)$ and $U=\mathfrak
u_q(\mathfrak{sl}_n)$ with $n>2$ (since for the rank $1$ case, there
are sufficient discussions in the literature).

\section{Loewy filtration of $\mathcal{A}_q^{(s)}(n, \bold m)$ and its rigidity}

\noindent {\it 3.1. Truncated objects $\mathcal{A}_q(n, \bold m)$.}
Set $\bold m=(m\ell{-}1, \cdots, m\ell{-}1)\in\mathbb{Z}_+^n$,
$m\in\mathbb{N}$, and
\begin{gather*}
\mathcal{A}_q(n, \bold m):=\span_k\bigl\{\,
x^{(\alpha)}\in\mathcal{A}_q(n)\,\,\big|\,\,\alpha\leq\bold m\,\bigr\}, \\
\mathcal{A}_q^{(s)}(n, \bold m):=\span_k\bigl\{\,
x^{(\alpha)}\in\mathcal{A}_q(n, \bold
m)\,\,\big|\,\,|\,\alpha\,|=s\,\bigr\},
\end{gather*}
then $\mathcal{A}_q(n, \bold
m)=\bigoplus_{s=0}^N\mathcal{A}_q^{(s)}(n, \bold m)$, where
$N=|\,\bold m\,|=n(m\ell{-}1)$.
\begin{prop}
$(1)$ $\mathcal{A}_q^{(s)}(n, \bold m)$ $(0\leq s\leq N)$ are
$\mathfrak{u}_q(\mathfrak{sl}_n)$-submodules.

$(2)$ $\dim\mathcal{A}^{(s)}_q(n, \bold m)=$ the coefficient of
$\,t^s$ of polynomial $P_{n, m\ell}(t)$
$=(1+t+t^2+\cdots+t^{m\ell-1})^{n}=\sum_{i=0}^{\lfloor
\frac{s}{m\ell}\rfloor}(-1)^i\binom{n}{i}\binom{n{+}s{-}im\ell{-}1}{n{-}1}$.
\end{prop}
\begin{proof} (1)
For any $x^{(\alpha)}\in \mathcal{A}_q^{(s)}(n, \bold m)$:

(i) if $\alpha_i=m\ell-1$, or $\alpha_i<m\ell-1$ and
$\alpha_{i+1}=0$, then Proposition 2.4, (2.1) yields
$e_i.\,x^{(\alpha)}=0$;

(ii) if $\alpha_i<m\ell-1$ and $\alpha_{i+1}>0$, then
$\alpha+\varepsilon_i-\varepsilon_{i+1}\le \bold m$, and
$x^{(\alpha+\varepsilon_i-\varepsilon_{i+1})}\in\mathcal{A}_q^{(s)}(n,
\bold m)$. Thus, $e_i.\,x^{(\alpha)}\in\mathcal{A}_q^{(s)}(n, \bold
m)$.

Similarly, Proposition 2.4, (2.2)--(2.4) imply $f_i.\,x^{(\alpha)},\
K_i^{\pm1}.\,x^{(\alpha)}\in\mathcal{A}_q^{(s)}(n, \bold m)$.

Hence, $\mathcal{A}_q^{(s)}(n, \bold m)$ is a
$\mathfrak{u}_q(\mathfrak{sl}_n)$-submodule.

\smallskip
(2) Note that $\bigl\{x^{(\alpha)}\in\mathcal{A}_q(n, \bold
m)\,\big|\,|\,\alpha\,|=s\bigr\}$ is a basis of
$\mathcal{A}_q^{(s)}(n,\bold m)$. The homomorphism $\phi:
\mathcal{A}_q(n)\longrightarrow k[\,t\,]$ with $\phi(x_i)=t,\
\phi(x_i^{(\ell)})=t^\ell$ restricted to $\mathcal{A}_q^{(s)}(n,
\bold m)$ counts up the cardinal of the above basis set as the
coefficients of $t^s$ of polynomial $P_{n,m\ell}(t)$. The final
identity is due to the expansion of generating function
$(1{-}t^{m\ell})^n(1{-}t)^{-n}$.
\end{proof}

\noindent {\it 3.2. Energy degrees and action rules.} In this
subsection, we introduce an important concept, the so-called {\it
``energy degree"}. We will see that this captures the essential
features of the submodule structures in the root of unity case.

For any rational number $x$, denote by $\lfloor x\rfloor$ the
integer part of $x$.

\begin{defi}
For any $x^{(\alpha)}\in\mathcal{A}_q(n, \bold m)$ or
$\mathcal{A}_q(n)$, the {\it energy degree} of $x^{(\alpha)}$,
denoted by $\mathrm{Edeg}\,x^{(\alpha)}$, is defined as
$$
\mathrm{Edeg}\,x^{(\alpha)}:=\sum_{i=1}^n\left\lfloor
\frac{\alpha_i}{\ell}\right\rfloor
=\sum_{i=1}^n\mathrm{Edeg}_i\,x^{(\alpha)},
$$
where $\mathrm{Edeg}_i\,x^{(\alpha)}$ indicates the {\it $i$-th
energy degree} of $x^{(\alpha)}$, i.e.,
$\mathrm{Edeg}_i\,x^{(\alpha)}:=\lfloor
\frac{\alpha_i}{\ell}\rfloor$.

In general, for any $x\in\mathcal{A}_q(n, \bold m)$ or
$\mathcal{A}_q(n)$, define
$$
\mathrm{Edeg}\,x:=\mathrm{max}\bigl\{\,\mathrm{Edeg}\,x^{(\alpha)}\,\bigl|\,x
=\sum k_\alpha x^{(\alpha)},\ k_\alpha\in k^* \,\bigr\}.
$$
\end{defi}

\begin{prop}
$\mathrm{Edeg}\,(u.\, x^{(\alpha)})\le \mathrm{Edeg}\,x^{(\alpha)}$,
for any $u\in \mathfrak{u}_q(\mathfrak{sl}_n)$ and
$x^{(\alpha)}\in\mathcal{A}_q^{(s)}(n, \bold m)$ or
$\mathcal{A}_q^{(s)}(n)$. In particular, $\mathrm{Edeg}_i\,(u.\,
x^{(\alpha)})\leq \mathrm{Edeg}_i\,x^{(\alpha)}$ for each $i$.
\end{prop}
\begin{proof}  It suffices to check the behavior of generators
$e_i, f_i, K_i^{\pm1}~(1\leq i<n)$ of
$\mathfrak{u}_q(\mathfrak{sl}_n)$ acting on any basis element
$x^{(\alpha)}\in\mathcal{A}_q^{(s)}(n, \bold m)$ or
$\mathcal{A}_q^{(s)}(n)$.

\smallskip
(1) Note that
$e_i.\,x^{(\alpha)}=[\alpha_i{+}1]\,x^{(\alpha+\varepsilon_i-\varepsilon_{i+1})}$,
by Proposition 2.4, (2.1).

If $e_i.\,x^{(\alpha)}\ne 0$, then $\ell\nmid(\alpha_i+1)$ and
$\alpha_{i+1}>0$. Observing
$$
\mathrm{Edeg}\,x^{(\alpha+\varepsilon_i-\varepsilon_{i+1})}\\
=\left\lfloor\frac{\alpha_1}{\ell}\right\rfloor+\cdots +\left\lfloor
\frac{\alpha_{i-1}}{\ell}\right\rfloor+\left\lfloor\frac{\alpha_i{+}1}{\ell}\right\rfloor
+\left\lfloor \frac{\alpha_{i+1}{-}1}{\ell}\right\rfloor+
\cdots+\left\lfloor \frac{\alpha_n}{\ell}\right\rfloor,
$$
we get
$$
\mathrm{Edeg}\,x^{(\alpha+\varepsilon_i-\varepsilon_{i+1})}-\mathrm{Edeg}\,x^{(\alpha)}
=\left\lfloor\frac{\alpha_i{+}1}{\ell}\right\rfloor+
\left\lfloor\frac{\alpha_{i+1}{-}1}{\ell}\right\rfloor-
  \left\lfloor\frac{\alpha_i}{\ell}\right\rfloor- \left\lfloor\frac{\alpha_{i+1}}{\ell}\right\rfloor.
$$

Obviously,
$\lfloor\frac{\alpha_{i+1}-1}{\ell}\rfloor-\lfloor\frac{\alpha_{i+1}}{\ell}\rfloor\leq0
$. If
$\lfloor\frac{\alpha_i+1}{\ell}\rfloor-\lfloor\frac{\alpha_i}{\ell}\rfloor>0
$, then $\alpha_i\equiv \ell{-}1 (\text{mod}\, \ell)$. It is
contrary to the assumption above, so
$\lfloor\frac{\alpha_i{+}1}{\ell}\rfloor-\lfloor\frac{\alpha_i}{\ell}\rfloor\leq
0$.

Therefore, $\mathrm{Edeg}\,(e_i.\,x^{(\alpha)})\leq
\mathrm{Edeg}\,x^{(\alpha)}$.

Similarly, by Proposition 2.4, (2.2)--(2.4), we get
$\mathrm{Edeg}\,(f_i.\,x^{(\alpha)})\leq
 \mathrm{Edeg}\,x^{(\alpha)}$,
$\mathrm{Edeg}\,(K_i^{\pm1}.\,x^{(\alpha)})=\mathrm{Edeg}\,x^{(\alpha)}$.

\smallskip
(2) In the proof of (1), we actually show that
$\mathrm{Edeg}_j\,(u.\, x^{(\alpha)})\leq
\mathrm{Edeg}_j\,x^{(\alpha)}$ for each $j$ and for arbitrary $u\in
\mathfrak{u}_q(\mathfrak{sl}_n)$,
$x^{(\alpha)}\in\mathcal{A}_q^{(s)}(n, \bold m)$ or
$\mathcal{A}_q^{(s)}(n)$.
\end{proof}

\begin{prop}
Given $x^{(\alpha)},x^{(\beta)}\in \mathcal{A}_q^{(s)}(n, \bold m)$
or $\mathcal{A}_q^{(s)}(n)$ with
$\mathrm{Edeg}\,x^{(\alpha)}=\mathrm{Edeg}\,x^{(\beta)}$. If
$\mathrm{Edeg}_i\,x^{(\alpha)}\ne\mathrm{Edeg}_i\,x^{(\beta)}$ for
some $i\ (1\leq i\leq n)$, then for any $u,\, v\in
\mathfrak{u}_q(\mathfrak{sl}_n)$, $u.\,x^{(\alpha)}\neq
x^{(\beta)}$, $v.\,x^{(\beta)}\neq x^{(\alpha)}$. Namely,
$x^{(\alpha)}\not\in\mathfrak{u}_q(\mathfrak{sl}_n).\,x^{(\beta)}$,
$x^{(\beta)}\notin\mathfrak{u}_q(\mathfrak{sl}_n).\,x^{(\alpha)}$.
\end{prop}
\begin{proof}  Without
loss of generality, we assume that
$\mathrm{Edeg}_i\,x^{(\alpha)}>\mathrm{Edeg}_i\,x^{(\beta)}$. Since
$\mathrm{Edeg}\,x^{(\alpha)}=\mathrm{Edeg}\,x^{(\beta)}$, there must
exist a $j\ne i$ with $1\leq j\leq n$ such that
$\mathrm{Edeg}_j\,x^{(\alpha)}<\mathrm{Edeg}_j\,x^{(\beta)}$.

\smallskip
(1) If there exists $u\in \mathfrak{u}_q(\mathfrak{sl}_n)$ such that
$u.\, x^{(\alpha)}= x^{(\beta)}$, by Proposition 3.3,
$\mathrm{Edeg}_r\,x^{(\beta)}=\mathrm{Edeg}_r\, (u.\,
x^{(\alpha)})\leq \mathrm{Edeg}_r\,x^{(\alpha)}$ for $1\le r\le n$.
It contradicts the fact
$\mathrm{Edeg}_j\,x^{(\alpha)}<\mathrm{Edeg}_j\,x^{(\beta)}$ for
some $j\,(\ne i)$. Hence, $ u.\, x^{(\alpha)}\neq x^{(\beta)}$, for
any $u\in \mathfrak{u}_q(\mathfrak{sl}_n)$.

\smallskip
(2) Using the assumption
$\mathrm{Edeg}_i\,x^{(\alpha)}>\mathrm{Edeg}_i\,x^{(\beta)}$, by a
similar argument of (1), we can derive $ v.\,x^{(\beta)}\neq
x^{(\alpha)}$, for any $v\in \mathfrak{u}_q(\mathfrak{sl}_n)$.
\end{proof}

The proof of Theorem 2.5 (2) (see \cite{HU}, 4.2) motivates the
following observation.

\begin{prop}
Given $x^{(\alpha)}, x^{(\beta)}\in\mathcal{A}_q^{(s)}(n, \bold m)$
or $\mathcal{A}_q^{(s)}(n)$ with
$\mathrm{Edeg}_i\,x^{(\alpha)}=\mathrm{Edeg}_i\,x^{(\beta)}$ for
each $i$ $(1\leq i \leq n)$, then there exist $u,\ v\in
\mathfrak{u}_q(\mathfrak{sl}_n)$ such that $u.\, x^{(\alpha)}=
x^{(\beta)}$,  $v.\, x^{(\beta)}= x^{(\alpha)}$. In this case,
$\mathfrak{u}_q(\mathfrak{sl}_n).\,x^{(\alpha)}=\mathfrak{u}_q(\mathfrak{sl}_n).\,x^{(\beta)}$.
\end{prop}
\begin{proof}
Put $m_j:=\mathrm{Edeg}_j\,x^{(\beta)}$ and
$r:=s-\sum_{j=1}^{n}m_j\ell$, for $x^{(\beta)}\in
\mathcal{A}_q^{(s)}(n, \bold m)$. Clearly, $0\leq r\leq
n(\ell{-}1)$. Write $r=(i{-}1)(\ell{-}1)+r_i$ with $1\le i\le n$,
$0\le r_i\le \ell{-}1$.

Set $\gamma:=({\ell{-}1, \cdots, \ell{-}1}, r_i,%\underbrace{\ell{-}1, \cdots, \ell{-}1}_{i-1}, r_i,
0, \cdots, 0)$ and $\eta:=(m_1\ell, \cdots,m_n\ell)+\gamma$. Then
$|\,\gamma\,|=r$, and $|\,\eta\,|=s$, that is, $x^{(\eta)}\in
\mathcal{A}_q^{(s)}(n, \bold m)$.

Write $\beta=\sum_{j=1}^n(m_j\ell+h_j)\varepsilon_j$ with $0\le
h_j\le \ell{-}1$, $\sum_{j=1}^nh_j=r$ (since $|\,\beta\,|=s$).
Denote by $k$ the last ordinal number with $h_k\ne 0$ for the
$n$-tuple $(h_1,\cdots,h_n)$. So, $k\geq i$ if $r_i\neq0$, and $k\ge
i{-}1$ if $r_i=0$.

\smallskip
(I) Note that the pair $(x^{(\eta)}, x^{(\beta)})$ satisfies the
hypothesis of our Proposition. Firstly, for the given pair
$(x^{(\eta)}, x^{(\beta)})$, we can prove the following Claims (A),
(B).

\smallskip
Claim (A): There exists $u_1\in \mathfrak u_q(\mathfrak{n}^-)$, such
that $u_1.\,x^{(\eta)}=x^{(\beta)}$.

\smallskip
Case (1). If $r_i\ge h_k$, then by Proposition 2.4, (2.2) \& Lemma
2.2 (1), we get
\begin{equation*}
\begin{aligned}
f_{k-1}^{h_k}\cdots f_i^{h_k}.\,x^{(\eta)}
&=\prod_{z=i+1}^k\prod_{j=1}^{h_k}[\,m_z\ell{+}j\,]\,x^{(\eta-h_k\varepsilon_i+h_k\varepsilon_k)}\\
&=q^{h_km_k\ell-\eta'*h_k\varepsilon_k}\prod_{z=i+1}^{k}\prod_{j=1}^{h_k}[\,m_z\ell{+}j\,]\,x^{(\eta')}x^{(h_k\varepsilon_k)}\neq0,
\end{aligned}
\end{equation*}
where $
\eta'=\eta-h_k\varepsilon_i=\sum_{j=1}^{i-1}(m_j\ell{+}\ell{-}1)\varepsilon_j+(m_i\ell{+}r_i{-}h_k)\varepsilon_i
+\sum_{j=i+1}^n(m_j\ell)\varepsilon_j$.

\smallskip
Case (2). If $r_i<h_k$, then by Proposition 2.4, (2.2) \& Lemma 2.2
(1), we get
\begin{equation*}
\begin{aligned}
f_{k-1}^{h_k-r_i}\cdots& f_{i-1}^{h_k-r_i}f_{k-1}^{r_i}\cdots
f_i^{r_i}.\,x^{(\eta)}\\
&\qquad
=\Bigl(\prod_{z=i+1}^k\prod_{j=1}^{r_i}[\,m_z\ell{+}j\,]\Bigr)\Bigl(\prod_{z'={i}}^{k-1}\prod_{j'=1}^{h_k-r_i}
[\,m_{z'}\ell{+}j'\,]\Bigr)\times \\
&\qquad\qquad \times[\,m_k\ell{+}r_i{+}1\,]\cdots
[\,m_k\ell{+}h_k\,]\,x^{(\eta-(h_k{-}r_i)\varepsilon_{i-1}-r_i\varepsilon_i+h_k\varepsilon_k)}\\
&\qquad=\Bigl(\prod_{z=i+1}^k\prod_{j=1}^{r_i}[\,m_z\ell{+}j\,]\Bigr)
\Bigl(\prod_{z'={i}}^{k-1}\prod_{j'=1}^{h_k-r_i}[\,m_{z'}\ell{+}j'\,]\Bigr)\times\\
&\qquad\qquad \times[\,m_k\ell{+}r_i{+}1\,]\cdots
[\,m_k\ell{+}h_k\,]\,q^{h_km_k\ell-\eta'*h_k\varepsilon_k}\,x^{(\eta')}x^{(h_k\varepsilon_k)}\neq0,
\end{aligned}
\end{equation*}
where $\eta'=\eta-(h_k{-}r_i)\varepsilon_{i-1}-r_i\varepsilon_i$.

Set $\beta':=\beta-h_k\varepsilon_k$. For the pair $(x^{(\eta')},
x^{(\beta')})$, using an induction on $\eta$ (at first, noting that
the argument holds for $\eta=\varepsilon_1=\lambda_1$), the same
argument of the proof of Theorem 2.5 (2) (see \cite{HU}, 4.2), there
exists $u_1'\in \mathfrak{u}_q(\mathfrak n^-)$ generated by $f_j\,
(j<k{-}1)$, such that $u_1'.\, x^{(\eta')}=x^{(\beta')}$. Note that
$f_j.\,x^{(h_k\varepsilon_k)}=0 \ (j<k{-}1)$ and $\Delta
f_j=f_j\otimes 1+K_j^{-1}\otimes f_j$, then
$$
u_1'.\,(x^{(\eta')}x^{(h_k\varepsilon_k)})=(u_1'.\,x^{(\eta')})\,x^{(h_k\varepsilon_k)}
=x^{(\beta')}x^{(h_k\varepsilon_k)}=q^{\beta'*h_k\varepsilon_k-m_kh_k\ell}\,x^{(\beta)}\neq0.
$$
Combining with both cases (1) and (2), we get the claim as desired.

\smallskip
Conversely, for the given pair $(x^{(\beta)}, x^{(\eta)})$, we can
prove the following

\smallskip
Claim (B):  There exists $u_2\in \mathfrak u_q(\mathfrak{n}^+)$,
such that $u_2.\,x^{(\beta)}=x^{(\eta)}$.

\smallskip
Case (i). If $r_i\geq h_k$, then
\begin{equation*}
e_i^{h_k}\cdots e_{k-1}^{h_k}.\, x^{(\eta'+h_k\varepsilon_k)}
=\Bigl(\prod_{z=i+1}^{k-1}\prod_{j=1}^{h_k}[m_z\ell{+}j]\Bigr)[m_i\ell{+}r_i{-}h_k{+}1]\cdots[m_i\ell{+}r_i]\,x^{(\eta)}\neq0,
\end{equation*}
where
$\eta'=\eta-h_k\varepsilon_i=\sum_{j=1}^{i-1}(m_j\ell{+}\ell{-}1)\varepsilon_j+(m_i\ell{+}r_i{-}h_k)\varepsilon_i
+\sum_{j=i+1}^n(m_j\ell)\varepsilon_j$.

\smallskip
Case (ii). If $r_i<h_k$, then
\begin{equation*}
\begin{aligned}
e_i^{r_i}\cdots e_{k-1}^{r_i}&e_{i-1}^{h_k-r_i}\cdots
e_{k-1}^{h_k-r_i}.\, x^{(\eta'+h_k\varepsilon_k)}\\
=&\,\Bigl(\prod_{z=i}^{k-1}\prod_{j=1}^{h_k-r_i}[m_z\ell{+}j]\Bigr)
[m_{i-1}\ell{+}\ell{-}1{-}(h_k{-}r_i){+}1]\cdots[m_{i-1}\ell{+}\ell{-}1]\times\\
&\times \Bigl(\prod_{z'=i}^{k-1}\prod_{j'=1}^{r_i}
[m_{z'}\ell{+}j']\Bigr)\,x^{(\eta)}\neq0,
\end{aligned}
\end{equation*}
where $\eta'=\eta-(h_k{-}r_i)\varepsilon_{i-1}-r_i\varepsilon_i$.

Inductively, for the pair $(\beta', \eta')$
$(\beta':=\beta\,{-}\,h_k\varepsilon_k)$, there exists $u_2'\in
\mathfrak{u}_q(\mathfrak n^+)$ generated by $e_j\, (j< k{-}1)$ such
that $u_2'.\,x^{(\beta')}=x^{(\eta')}$. Note that
$e_j.\,x^{(h_k\varepsilon_k)}=0$,
$K_j.\,x^{(h_k\varepsilon_k)}=x^{(h_k\varepsilon_k)}$ for $j<k{-1}$,
and $\Delta e_j=e_j\otimes K_j+1\otimes e_j$, then there are $c,
\,c'\in k^*$ such that
$$
u_2'.\,
x^{(\beta)}=c\,u_2'.\,(x^{(\beta')}\,x^{(h_k\varepsilon_k)})=
c\,(u_2'.\,x^{(\beta')})\,x^{(h_k\varepsilon_k)}=c'\;x^{(\eta'+h_k\varepsilon_k)}\neq0.
$$
Combining with both cases (i) and (ii), we get the claim as
required.

\smallskip
(II) For the given pair $(x^{(\alpha)}, x^{(\beta)})$ satisfying the
hypothesis of our Proposition, consider both pairs $(x^{(\eta)},
x^{(\beta)})$ and $(x^{(\alpha)}, x^{(\eta)})$, by Claims (A) and
(B), we see that there exists $u_1, u_2\in
\mathfrak{u}_q(\mathfrak{sl}_n)$ such that $u_1.\,
x^{(\eta)}=x^{(\beta)}, u_2.\, x^{(\alpha)}=x^{(\eta)}$. Set
$u=u_1u_2$, then $u.\, x^{(\alpha)}=x^{(\beta)}$. Similarly, there
exists $v\in \mathfrak{u}_q(\mathfrak{sl}_n)$ such that $v.\,
x^{(\beta)}=x^{(\alpha)}$.
\end{proof}

The observation below is more crucial to understand the submodules
structure of $\mathcal{A}_q^{(s)}(n, \bold m)$ and
$\mathcal{A}_q^{(s)}(n)$. Its proof is skillful.

\begin{prop}
Given $x^{(\alpha)},x^{(\beta)}\in \mathcal{A}_q^{(s)}(n, \bold m)$
or $\mathcal{A}_q^{(s)}(n)$ with $\mathrm{Edeg}
\,x^{(\alpha)}>\mathrm{Edeg}\,x^{(\beta)}$. If $\mathrm{Edeg}_i
\,x^{(\alpha)}\geq\mathrm{Edeg}_i\,x^{(\beta)}$ for each $i$, then
there exists $u\in \mathfrak{u}_q(\mathfrak{sl}_n)$ such that $u.\,
x^{(\alpha)}=x^{(\beta)}$. That is,
$\mathfrak{u}_q(\mathfrak{sl}_n).\,x^{(\beta)}\subsetneq\mathfrak{u}_q(\mathfrak{sl}_n).\,x^{(\alpha)}$.
\end{prop}
\begin{proof} (I) Assume $\mathrm{Edeg}\,
x^{(\alpha)}=\mathrm{Edeg}\,x^{(\beta)}+1$. Then there exists $j\,
(1\leq j\leq n)$ such that
$\mathrm{Edeg}_j\,x^{(\alpha)}=\mathrm{Edeg}_j\,x^{(\beta)}+1$ and
$\mathrm{Edeg}_i\,x^{(\alpha)}=\mathrm{Edeg}_i\,x^{(\beta)}$ for
$i\ne j$.

Write $\alpha=\sum_{i=1}^{n}(m_i\ell+r_i)\varepsilon_i$, where
$m_i=\mathrm{Edeg}_i\,x^{(\alpha)}$ and $0\leq r_i\leq \ell{-}1$,
then $0\leq \sum_{i=1}^{n}r_i\leq n(\ell{-}1)$. Note that
$|\,\alpha\,|=|\,\beta\,|=s=\ell\cdot\mathrm{Edeg}\,x^{(\alpha)}+\sum_{i=1}^{n}r_i$.
By the assumption above, we must have
$\sum_{i=1}^{n}r_i<(n{-}1)(\ell{-}1)$. Otherwise,
$\sum_{i=1}^{n}r_i\geq(n{-}1)(\ell{-}1)$. This implies that
$\mathrm{Edeg}\,x^{(\alpha)}$ is the least among the
$\mathrm{Edeg}\,x^{(\theta)}$'s, for any
$x^{(\theta)}\in\mathcal{A}_q^{(s)}(n, \bold m)$. It contradicts the
given condition $\mathrm{Edeg}\,
x^{(\alpha)}>\mathrm{Edeg}\,x^{(\beta)}$.

\smallskip
(1) When $j<n$: as $\sum_{i=1}^{n}r_i<(n{-}1)(\ell{-}1)$, there
exists $(h_1,\cdots,h_n)\in\mathbb{Z}^n_+$ with $h_j=0$,
$h_{j+1}<\ell{-}1$, $0\leq h_i\leq \ell{-}1$ for $i\ne j, j{+}1$
such that $\sum_{i=1}^nh_i=\sum_{i=1}^{n}r_i$.

Set $\gamma=\sum_{i=1}^n(m_i\ell+h_i)\varepsilon_i$, then
$|\,\gamma\,|=s$, i.e., $x^{(\gamma)}\in \mathcal{A}_q^{(s)}(n,
\bold m)$. Obviously,
$\mathrm{Edeg}_i\,x^{(\alpha)}=\mathrm{Edeg}_i\,x^{(\gamma)}$ for
each $i$.

Again, we have $f_j.\,
x^{(\gamma)}=[\,m_{j+1}\ell{+}h_{j+1}{+}1\,]\,x^{(\gamma-\varepsilon_j+\varepsilon_{j+1})}\,(\neq0)\in
\mathcal{A}_q^{(s)}(n, \bold m)$,
$\mathrm{Edeg}_j(f_j.\,x^{(\gamma)})=m_j{-}1=\mathrm{Edeg}_j\,x^{(\beta)}$
and
$\mathrm{Edeg}_i(f_j.\,x^{(\gamma)})=\mathrm{Edeg}_i\,x^{(\beta)}$,
for $i\ne j$.

Hence, for the pairs $(x^{(\alpha)}, x^{(\gamma)})$ and
$(f_j.\,x^{(\gamma)}, x^{(\beta)})$, by Proposition 3.5, there exist
$u_1, \,u_2\in \mathfrak{u}_q(\mathfrak{sl}_n)$ such that $u_1.\,
x^{(\alpha)}=x^{(\gamma)}$ and
$u_2.\,(f_j.\,x^{(\gamma)})=x^{(\beta)}$. Set $u=u_2f_ju_1$, then
$u.\, x^{(\alpha)}=x^{(\beta)}$.

\smallskip
(2) When $j=n$: as $\sum_{i=1}^{n}r_i<(n{-}1)(\ell{-}1)$, there
exists $(h_1',\cdots,h_n')\in\mathbb{Z}^n_+$ with
$h'_{n-1}<\ell{-}1$, $h_n'=0$, $0\leq h'_i\leq \ell{-}1$ for $i\ne
n{-}1, n$ such that $\sum_{i=1}^nh_i'=\sum_{i=1}^{n}r_i$.

Set $\gamma'=\sum_{i=1}^n(m_i\ell+h_i')\varepsilon_i$, then
$|\,\gamma'\,|=s$, i.e., $x^{(\gamma')}\in \mathcal{A}_q^{(s)}(n,
\bold m)$. Obviously,
$\mathrm{Edeg}_i\,x^{(\alpha)}=\mathrm{Edeg}_i\,x^{(\gamma')}$ for
each $i$.

Again,
$e_{n-1}.\,x^{(\gamma')}=[\,m_{n-1}\ell{+}h'_{n-1}{+}1\,]\,x^{(\gamma'+\varepsilon_{n-1}-\varepsilon_n)}\,(\neq0)\in
\mathcal{A}_q^{(s)}(n, \bold m)$, and
$\mathrm{Edeg}_i(e_{n-1}.\,x^{(\gamma')})=\mathrm{Edeg}_i\,x^{(\beta)}$
for $i<n$, $
\mathrm{Edeg}_n(e_{n-1}.\,x^{(\gamma')})=m_n{-}1=\mathrm{Edeg}_n\,x^{(\beta)}$.

Now for the pairs $(x^{\alpha}, x^{(\gamma')})$ and $(e_{n-1}.\,
x^{(\gamma')}, x^{(\beta)})$, using Proposition 3.5, there exists
$v_1,\, v_2\in \mathfrak{u}_q(\mathfrak{sl}_n)$ such that $v_1.\,
x^{(\alpha)}=x^{(\gamma')}$ and $v_2.\,(e_{n-1}.\,
x^{(\gamma')})=x^{(\beta)}$. Set $u=v_2e_{n-1}v_1$, then $u.\,
x^{(\alpha)}=x^{(\beta)}$.

\smallskip
(II) Use an induction on
$\mathrm{Edeg}\,x^{(\alpha)}-\mathrm{Edeg}\,x^{(\beta)}$. As
$\mathrm{Edeg}\, x^{(\alpha)}>\mathrm{Edeg}\,x^{(\beta)}$, according
to the proof of (I), it is clear that there are $x^{(\gamma_1)}\in
\mathcal{A}_q^{(s)}(n, \bold m)$ with
$\mathrm{Edeg}\,x^{(\alpha)}=\mathrm{Edeg}\,x^{(\gamma'_1)}+1$ and
$\mathrm{Edeg}_i\,x^{(\alpha)}\ge\mathrm{Edeg}_i\,x^{(\gamma_1')}$
for each $i$, and $u_1\in \mathfrak{u}_q(\mathfrak{sl}_n)$ such that
$u_1.\,x^{(\alpha)}=x^{(\gamma'_1)}$. And for the pair
$(x^{(\gamma'_1)}, x^{(\beta)})$, by the inductive hypothesis, there
is $u_2\in\mathfrak{u}_q(\mathfrak{sl}_n)$ such that
$u_2.\,x^{(\gamma'_1)}=x^{(\beta)}$.

This completes the proof.
\end{proof}

\noindent {\it 3.3. Equivalence and ordering on $n$-tuples.} Note
that the set of $n$-tuples of nonnegative integers indexes a basis
of $\mathcal A_q(n)$ via the mapping $\chi: \mathbb
Z_+^n\longrightarrow \mathcal A_q(n)$ such that
$\chi(\alpha)=x^{(\alpha)}$. Set $\mathbb
Z_+^n(s):=\{\alpha\in\mathbb Z_+^n\mid |\,\alpha\,|=s\}$, $\mathbb
Z_+^n(s,\bold m):=\{\alpha\in\mathbb Z_+^n(s)\mid \alpha\le \bold
m\}$. These index bases of $\mathcal A_q^{(s)}(n)$ and $\mathcal
A_q^{(s)}(n,\bold m)$, respectively.

Set $\mathcal E_i(\alpha):=\mathrm{Edeg}_i\,x^{(\alpha)}$ and
$\mathcal E(\alpha):=(\mathcal E_1(\alpha),\cdots,\mathcal
E_n(\alpha))$. Define an equivalence $\thicksim$ on $\mathbb
Z_+^n(s,\bold m)$ or $\mathbb Z_+^n(s)$ as follows:
$\alpha\thicksim\beta$ $\Longleftrightarrow$ $\mathcal
E(\alpha)=\mathcal E(\beta)$, for any $\alpha, \beta\in\mathbb
Z_+^n(s,\bold m)$ or $\mathbb Z_+^n(s)$. So, Proposition 3.4 shows
that $\alpha\not\thicksim\beta\in\mathbb Z_+^n(s,\bold m)$ or
$\mathbb Z_+^n(s)$, then
$x^{(\alpha)}\not\in\mathfrak{u}_q(\mathfrak{sl}_n).\,x^{(\beta)}$
and
$x^{(\beta)}\notin\mathfrak{u}_q(\mathfrak{sl}_n).\,x^{(\alpha)}$.
While, Proposition 3.5 indicates that if
$\alpha\thicksim\beta\in\mathbb Z_+^n(s,\bold m)$ or $\mathbb
Z_+^n(s)$, then
$\mathfrak{u}_q(\mathfrak{sl}_n).\,x^{(\alpha)}=\mathfrak{u}_q(\mathfrak{sl}_n).\,x^{(\beta)}$.

Introduce an ordering $\succeq$ on $\mathbb Z_+^n$ as follows:
$\alpha\succeq\beta$ $\Longleftrightarrow$ $\mathcal E(\alpha)\ge
\mathcal E(\beta)$ $\Longleftrightarrow$ $\mathcal E_i(\alpha)\ge
\mathcal E_i(\beta)$ for each $i$. So, Proposition 3.6 means that if
$\alpha\succeq\beta\in\mathbb Z_+^n(s,\bold m)$ or $\mathbb
Z_+^n(s)$, then
$\mathfrak{u}_q(\mathfrak{sl}_n).\,x^{(\beta)}\subsetneq\mathfrak{u}_q(\mathfrak{sl}_n).\,x^{(\alpha)}$.
Actually, Proposition 3.6 captures an essential feature between the
ordering relations $\succeq$ on the set of $n$-tuples of
energy-degrees $\{\mathcal E(\alpha)\}$ and the including relations
of submodules of $\mathcal{A}_q^{(s)}(n, \bold m)$ or
$\mathcal{A}_q^{(s)}(n)$. This will be useful to analyse their
indecomposability.

%For any $x^{(\alpha)},x^{(\beta)}\in \mathcal{A}_q^{(s)}(n, \bold
%m)$, denote by $x^{(\alpha)}\thicksim x^{(\beta)}$ if
%$\mathrm{Edeg}_i\,x^{(\alpha)}=\mathrm{Edeg}_i\,x^{(\beta)}$ for
%$i=1,\cdots, n$. It is an equivalence relation. Then the basis set
%$\langle x^{(\alpha)}\in\mathcal{A}_q^{(s)}(n, \bold
%m)\mid\alpha\in\mathbb{Z}^n \rangle$ of $\mathcal{A}_q^{(s)}(n,
%\bold m)$ can be divided into different equivalent classes.
%According to Propositions, by the actions of
%$\mathfrak{u}_q(\mathfrak{sl}_n)$ the elements in the same
%equivalent class can be transferred into each other, and the
%elements in different equivalent classes can be transferred from one
%to another if and only if they satisfy the conditions in Proposition
%2.9.

\smallskip
\noindent {\it 3.4. Socle of $\mathcal{A}_q^{(s)}(n, \bold m)$.}
Given $0\le s\le N$ $(N=|\,\bold m\,|$), denote by $E(s)_0$ (resp.
$E(s)$) the lowest (resp. highest) energy degree of elements of
$\mathcal{A}_q^{(s)}(n, \bold m)$.

The following observation will be essential to describing the whole
picture of the submodules structure of $\mathcal{A}_q^{(s)}(n, \bold
m)$ in a more explicit manner.
\begin{lemm} Suppose $n\geq3$ and $\mathbf{char}(q)=\ell\ge 3$.
Given $s$ with $0\le s\le N$, where $N=|\,\bold
m\,|=n(m\ell{-}1)\,$.

\smallskip
$(1)$ When $0\le s\le \ell{-}1:$ \ $E(s)_0=0=E(s)$.

\smallskip
$(2)$ When $(\ell{-}1){+}1\le s\le n(\ell{-}1):$ \ $E(s)_0=0$, and
$\,1\,\;{\le}\,\; E(s)\,\;{\le}\,\; E\bigl(n(\ell{-}1)\bigr)$, where
$n=n'\ell{+}r \ (0\le r<\ell)$,
$E\bigl(n(\ell{-}1)\bigr)=n{-}n'{-}1{+}\delta_{n,n'\ell}$. More
precisely,

\smallskip
$E(s)=j{-}j'{-}\sum_{i=1}^{\ell{-}1}\delta_{i,r_j-h}$, for
$s=j(\ell{-}1)+h$ with $0\le h\le\ell{-}1$, $1\le j\le n$, where
$j=j'\ell{+}r_j \ (0\le r_j<\ell)$. Namely,

\smallskip
$j{-}j'{-}1\le E(s)\le j{-}j'$, for $j(\ell{-}1){+}1\le s\le
(j{+}1)(\ell{-}1)$ with $1\le j\le n{-}1$.

\smallskip
$(3)$ When $n(\ell{-}1){+}1\le s\le N{-}\ell:$ \ $E(s)_0=k$, and
$\,k{+}1\le E(s)\le n(m{-}1)$, for $s=k\ell{+}h{+}(n{-}1)(\ell{-}1)$
with $0\le h\le\ell{-}1$ and $1\le k\le n(m{-}1){-}1$. More
precisely,

\smallskip
$k{+}E\bigl((n{-}1)(\ell{-}1)\bigr)\le E(s)\le
k{+}E\bigl(n(\ell{-}1)\bigr)$, for $k\le
n(m{-}1){-}E\bigl(n(\ell{-}1)\bigr);$ and $E(s)=n(m{-}1)$, for
$k>n(m{-}1){-}E\bigl(n(\ell{-}1)\bigr)$, \\
where
$E\bigl((n{-}1)(\ell{-}1)\bigr)=n{-}n'{-}1-\sum_{i=2}^{\ell{-}1}\delta_{i,n-n'\ell}\ge
1$ under the assumption $n\ge 3$.

\smallskip
$(4)$ When $N{-}(\ell{-}1)\le s\le N:$ \ $E(s)_0=n(m{-}1)=E(s)$.
\end{lemm}
\begin{proof}
Given any $\alpha\in \mathbb Z_+^n(s,\bold m)$, denote by
$\gamma(\alpha):=\alpha-\ell\cdot\mathcal
E(\alpha)=(r_1,\cdots,r_n)$ the rest $n$-tuple of $\alpha$ with
respect to its energy-degree $n$-tuple. Clearly,
$\gamma(\alpha)\le\tau$.

Now from the definitions of $E(s)_0$ and $E(s)$, there are at least
$\alpha, \beta\in \mathbb Z_+^n(s,\bold m)$ such that
$E(s)_0=|\,\mathcal E(\alpha)\,|$ and $E(s)=|\,\mathcal
E(\beta)\,|$, as well as $s=|\,\alpha\,|=\ell\cdot
E(s)_0+|\,\gamma(\alpha)\,|$ with
$|\,\gamma(\alpha)\,|=\sum_{i=1}^nr_i$ as largest as possible,
 and
$s=|\,\beta\,|=\ell\cdot E(s)+|\,\gamma(\beta)\,|$ with
$|\,\gamma(\beta)\,|$ as smallest as possible.

\smallskip
Based on the above observation, the conclusion (1) is clear. As for
(4), we note that for any $x^{(\alpha)}\in\mathcal{A}_q^{(s)}(n,
\bold m)$, $\alpha$ is of the form
$\bigl((m{-}1)\ell{+}a_1,\cdots,(m{-}1)\ell{+}a_n\bigr)$ with
$\gamma(\alpha)=(a_1,\cdots,a_n)\le \tau$ such that
$|\,\gamma(\alpha)\,|=(n{-}1)(\ell{-}1)+h$ with $0\le h\le\ell{-}1$,
and $\mathcal E(\alpha)=\bigl(m{-}1,\cdots,m{-}1)$. So,
$E(s)_0=E(s)=n(m{-}1)$.

\smallskip
$(2)$ When $\ell\le s\le n(\ell{-}1)$: it is clear that $E(s)_0=0$,
as even for the extreme case $s=n(\ell{-}1)$, taking $\alpha=\tau$,
we get that $s=|\,\tau\,|$, $\gamma(\tau)=\tau$ and $\mathcal
E(\tau)=\bold 0$, i.e., $E(s)_0=0$.

In order to estimate $E(s)$, now we can assume that
$j(\ell{-}1){+}1\le s\le (j{+}1)(\ell{-}1)$ for $1\le j\le n{-}1$.
Let us consider the general case $s=j(\ell{-}1)+h$ with $0\le
h<\ell$ and $1\le j\le n$. Write $j=j'\ell{+}r_j$ with $0\le
r_j<\ell$. Then rewrite $s=j(\ell{-}1)+h=
\bigl(j'(\ell{-}1)+r_j\bigr)\ell-(r_j-h)$. Clearly, when $h\ge r_j$,
$E(s)=j'(\ell{-}1)+r_j=j-j'$; and when $h<r_j$,
$E(s)=j'(\ell{-}1)+r_j-1=j-j'-1$. Particularly, when $s=j(\ell{-}1)$
with $h=0$, we get $E(s)=j{-}j'{-}1{+}\delta_{j,j'\ell}$. So we
obtain $j{-}j'{-}1\le E(s)\le j{-}j'$, for $j(\ell{-}1){+}1\le s\le
(j{+}1)(\ell{-}1)$ with $1\le j\le n{-}1$.

\smallskip
(3) When $n(\ell{-}1){+}1\le s\le N{-}\ell$: Firstly, we rewrite
$N-\ell=n(m\ell{-}1)-\ell=\bigl(n(m{-}1){-}1\bigr)\ell{+}n(\ell{-}1)$.
So now for the $s$ given above, we can put it into a certain
strictly smaller interval:
$k\ell{+}(n{-}1)(\ell{-}1)=(k{-}1)\ell{+}1{+}n(\ell{-}1)\le s\le
k\ell{+}n(\ell{-}1)$, for some $k$ with $1\le k\le n(m{-}1){-}1$.
Namely, $s=k\ell{+}h{+}(n{-}1)(\ell{-}1)$ with $0\le h\le\ell{-}1$.

Secondly, write $k=k'n{+}r$ $(0\le r<n)$. Note
$n(m{-}1){-}1=(m{-}2)n{+}(n{-}1)$. Taking
$\alpha=(\underbrace{k'{+}1,\cdots,k'{+}1}_{r},k',\cdots,k')\ell+(h,\ell{-}1,\cdots,\ell{-}1)
$, we obtain $|\,\alpha\,|=s$, i.e., $\alpha\in \mathbb
Z_+^n(s,\bold m)$, $\mathcal
E(\alpha)=(k'{+}1,\cdots,k'{+}1,k',\cdots,k')$,
$\gamma(\alpha)=(h,\ell{-}1,\cdots,\ell{-}1)$ with
$|\,\gamma(\alpha)\,|=(n{-}1)(\ell{-}1){+}h$ large enough. So,
$E(s)_0=|\,\mathcal E(\alpha)\,|=k$.

Finally, as for the estimate of $E(s)$, for
$k\ell{+}(n{-}1)(\ell{-}1)\le s\le k\ell{+}n(\ell{-}1)$, in view of
(2), from $n=n'\ell{+}r$, we get that $(n{-}1)'=n'{-}1$ if $r=0$,
and $(n{-}1)'=n'$ if $r>0$. Therefore,
$E((n{-}1)(\ell{-}1))=(n{-}1){-}(n{-}1)'{-}1{+}\delta_{n-1,(n-1)'\ell}=n{-}n'{-}1$
if $r=0$, $1$; and $E((n{-}1)(\ell{-}1))=n{-}n'{-}2$ if $r>1$. So,
for the above $s$, we get
$$
k+\Bigl(n{-}n'{-}1{-}\sum_{i=2}^{\ell{-}1}\delta_{i,n-n'\ell}\Bigr)\le
E(s)\le k+\Bigl(n{-}n'{-}1{+}\delta_{n,n'\ell}\Bigr),
$$
only if $k+(n{-}n'{-}1{+}\delta_{n,n'\ell})\le n(m{-}1)$. Otherwise,
$E(s)=n(m{-}1)$.

This completes the proof.
\end{proof}

\begin{theorem}
Assume that $n\geq3$ and $\mathbf{char}(q)=\ell\ge 3$. Then for the
$\mathfrak{u}_q(\mathfrak{sl}_n)$-modules $\mathcal{A}_q^{(s)}(n,
\bold m)$ with $0\le s\le N$, one has

$(1)$ For any nonzero $y\in\mathcal{A}_q^{(s)}(n, \bold m)$ with
energy degree $\mathrm{Edeg}\,(y)$, assume that the submodule
$\mathfrak V_y=\mathfrak{u}_q(\mathfrak{sl}_n).\,y$ is simple, then
$\mathrm{Edeg}\,(y)=E(s)_0$.

$(2)$  $\mathrm{Soc}\,\mathcal{A}_q^{(s)}(n, \bold m)=\span_k\{\,
x^{(\alpha)}\in \mathcal{A}_q^{(s)}(n, \bold m)\,\big|\, |\,\mathcal
E(\alpha)|=E(s)_0\,\}$.

$(3)$ $\mathcal{A}_q^{(s)}(n, \bold m)=\sum_{\alpha\in\mathbb
Z_+^n(s,\bold m): \,|\mathcal E(\alpha)|=E(s)}\mathfrak V_\alpha$,
where $\mathfrak
V_\alpha=\mathfrak{u}_q(\mathfrak{sl}_n).\,x^{(\alpha)}$.

$(4)$ When $0\le s\le \ell{-}1$, or $N{-}(\ell{-}1)\le s\le N:$
$\mathcal{A}_q^{(s)}(n, \bold m)=\mathfrak V_{\eta}$ is simple,
where $\eta=(s,0,\cdots,0)$ for $0\le s<\ell$, or
$\eta=(m\ell{-}1,\cdots,m\ell{-}1,(m{-}1)\ell{+}h)$ with
$s=|\,\eta\,|=n(m{-}1)\ell{+}(n{-}1)(\ell{-}1){+}h$, $(1\le
h\le\ell{-}1)$, and $x^{(\eta)}$ is the respective highest weight
vector.

$(5)$ When $\ell\le s\le N{-}\ell:$ $\mathcal{A}_q^{(s)}(n, \bold
m)$ is indecomposable. Moreover,

$(\text{\rm i})$ for $(\ell{-}1){+}1\le s\le n(\ell{-}1):$
$\mathrm{Soc}\,\mathcal{A}_q^{(s)}(n, \bold m)=\mathfrak V_{\eta}$
is simple, where $\eta=(\ell{-}1,\cdots,\ell{-}1,h,0,\cdots,0)$ with
$s=|\,\eta\,|=j(\ell{-}1){+}h$, $(1\le h\le\ell{-}1, \,1\le j<n)$,
and $x^{(\eta)}$ is the highest weight vector;

$(\text{\rm ii})$ for $n(\ell{-}1){+}1\le s\le N{-}\ell:$
$\mathrm{Soc}\,\mathcal{A}_q^{(s)}(n, \bold
m)=\bigoplus_{\eta(\b{\kappa})\in \wp}\mathfrak
V_{\eta(\b{\kappa})}$ is non-simple, where
$\eta(\b{\kappa})\in\wp=\{(\kappa_1\ell{+}(\ell{-}1),\cdots,\kappa_{n-1}\ell{+}(\ell{-}1),\kappa_n\ell{+}h)\mid
 \sum\kappa_i=\kappa$, $0\le \kappa_i\le m{-}1\}$ with
 $s=|\,\eta(\b{\kappa})\,|=\kappa\ell{+}h{+}(n{-}1)(\ell{-}1)$, $(1\le \kappa\le n(m{-}1){-}1$, $0\leq h\leq
 \ell{-}1)$, and $x^{(\eta(\b{\kappa}))}$'s are the respective highest weight
 vectors.
\end{theorem}
\begin{proof} (1) If $\mathrm{Edeg}\,(y)>E(s)_0$, then by Definition
3.2, in the expression of $y=\sum_{\alpha}k_\alpha x^{(\alpha)}$,
there exists some $\beta\in\mathbb Z_+^n(s,\bold m)$, $k_\beta\ne 0$
such that $|\,\mathcal E(\beta)\,|=\mathrm{Edeg}\,(y)$. By
Proposition 3.6, we can find $u\in\mathfrak{u}_q(\mathfrak{sl}_n)$
such that $u.\,x^{(\beta)}\ne0$ (then $u.\, y\neq 0$) but
$\mathrm{Edeg}\,(u.\,y)=\mathrm{Edeg}\,(u.\,x^{(\beta)})<\mathrm{Edeg}\,(y)$,
so we get a proper submodule $(0\ne)\, \mathfrak V_{u.\,y}\subsetneq
\mathfrak V_y$. It is a contradiction. So the assertion is true.

\smallskip
(2) follows from the conclusion (1), together with Propositions
3.4--3.6. Since for those $\alpha, \beta\in \mathbb Z_+^n(s, \bold
m)$ with $|\,\mathcal E(\alpha)\,|=|\,\mathcal E(\beta)\,|=E(s)_0$,
if $\alpha\sim\beta$, then $\mathfrak V_{x^{(\alpha)}}=\mathfrak
V_{x^{(\beta)}}$, by Proposition 3.5; and if $\alpha\nsim\beta$,
then by Propositions 3.4 \& 3.6, $\mathfrak
V_{x^{(\alpha)}}\cap\mathfrak V_{x^{(\beta)}}=0$.

\smallskip
(3) For any $\alpha\in\mathbb Z_+^n(s, \bold m)$ with
$|\,\alpha\,|=s$, according to the pre-ordering $\succeq$ defined in
subsection 3.3, we assert that there exists a $\varpi\in \mathbb
Z_+^n(s, \bold m)$ with $|\,\mathcal E(\varpi)\,|=E(s)$, such that
$\varpi\succeq\alpha$. Actually, this fact follows from the proof of
Lemma 3.7. Since $s=\ell\cdot|\,\mathcal
E(\alpha)\,|+|\,\gamma(\alpha)\,|=\ell E(s)+\hbar$ ($0\le \hbar\le
n(\ell{-}1)\,$), if $|\,\mathcal E(\alpha)\,|<E(s)$, writing
$|\,\gamma(\alpha)\,|=\kappa\ell+r$ $(0\le r<\ell)$, then
$\hbar=h\ell+r$ and $\kappa\ge \kappa-h= E(s)-|\,\mathcal E(\alpha)\,|$.
Construct $\mathcal E(\varpi)=(\mathcal
E_1(\alpha)+\jmath_1,\cdots,\mathcal E_n(\alpha)+\jmath_n)$ and
$\gamma(\varpi)=(\hbar_1,\cdots,\hbar_n)\le \tau$, such that each
$\mathcal E_i(\alpha)+\jmath_i\le m{-}1$, and
$\sum\jmath_i=\kappa-h$, $\sum\hbar_i=\hbar$. Taking
$\varpi=\ell\mathcal E(\varpi)+\gamma(\varpi)$, we get  $\varpi\le\bold
m$, $|\,\varpi\,|=s$, i.e., $\varpi\in\mathbb Z_+^n(s, \bold m)$ and
$|\,\mathcal E(\varpi)\,|=E(s)$, such that $\varpi\succeq \alpha$.

Again, from Proposition 3.6, together with its proof, there is $u\in
\mathfrak{u}_q(\mathfrak{sl}_n)$ such that
$u.\,x^{(\varpi)}=x^{(\alpha)}$. Hence, we arrive at the result as
stated.

\smallskip
(4) In these two extreme cases, by Lemma 3.7, we have $E(s)_0=E(s)$.
Note that the generating sets of (3) in both cases only contain one
equivalent class with respect to the equivalent relation $\sim$
defined in subsection 3.3. Thus, the above conclusions (2) \& (3)
give us the desired result below:
$$
\mathcal{A}_q^{(s)}(n, \bold m)=\mathrm{Soc}\,\mathcal{A}_q^{(s)}(n,
\bold m)=\mathfrak V_{\eta}
$$
is simple, here $\eta=(s,0,\cdots,0)$ for $0\le s<\ell$, or
$\eta=\bigl(m\ell{-}1,\cdots,m\ell{-}1,(m{-}1)\ell{+}h\bigr)$
 with $s=|\,\eta\,|=n(m{-}1)\ell{+}(n{-}1)(\ell{-}1){+}h$ $(0\le h\,{<}\,\ell)$, $x^{(\eta)}$ is the respective
 highest weight vector (by Theorem 2.5 (2), or
Proposition 3.5).

\smallskip
(5) By Lemma 3.7, (2) \& (3), we have $E(s)_0<E(s)$. Consequently,
(2) \& (3) give rise to the fact that
$\mathrm{Soc}\,\mathcal{A}_q^{(s)}(n, \bold
m)\subsetneq\mathcal{A}_q^{(s)}(n, \bold m)$.

\smallskip
(i) When $\ell\leq s\leq n(\ell{-}1)$: Due to Lemma 3.7, $E(s)_0=0$.
Then those $n$-tuples $\alpha\in \mathbb Z_+^n(s,\bold m)$ with
$|\,\mathcal E(\alpha)\,|=E(s)_0=0$ (namely, $\alpha\le \tau$) are
equivalent to each other with respect to $\sim$, and
$\eta=(\ell{-}1,\cdots,\ell{-}1,h,0,\cdots,0)$ is one of their
representatives, here $|\,\eta\,|=j(\ell{-}1){+}h=s$ ($1\leq
h\le\ell{-}1,\ 1\le j<n)$, i.e., $\eta\in\mathbb Z_+^n(s,\bold m)$.

Hence, $\mathrm{Soc}\,\mathcal{A}_q^{(s)}(n, \bold m)=\mathfrak
V_{\eta}$ is simple, where $x^{(\eta)}$ is the highest weight
vector, by Theorem 2.5 (2). Consequently, $\mathcal{A}_q^{(s)}(n,
\bold m)$ is indecomposable.

\smallskip
(ii) When $n(\ell{-}1){+}1\le s\le n(m\ell{-}1)\ell{-}\ell$: Thanks
to Lemma 3.7, we can set $s=\kappa\ell{+}h{+}(n{-}1)(\ell{-}1)$ with
$1\le \kappa\le n(m{-}1){-}1$ and $0\leq h\leq \ell{-}1$, then
$E(s)_0=\kappa$, and $\kappa{+}1\le E(s)\le n(m{-}1)$ under the
assumption $n>2$ (see Lemma 3.7).

Consider the set of equivalent classes of $n$-tuples $\eta\in\mathbb
Z_+^n(s,\bold m)$ with $s=|\,\eta\,|$ and $|\,\mathcal
E(\eta)\,|=E(s)_0=\kappa$. Denote it by $\wp$. Clearly, those
$\eta\in\wp$ can be constructed as follows: For the given $\kappa$,
set $\b{\kappa}=(\kappa_1,\cdots,\kappa_n)$ ($0\le \kappa_i\le
m{-}1$) with $\sum \kappa_i=\kappa$,
$\gamma=(\ell{-}1,\cdots,\ell{-}1,h)$ with
$|\,\gamma\,|=(n{-}1)(\ell{-}1){+}h$. Now take
$\eta:=\eta(\b{\kappa})=\ell\cdot\b{\kappa}+\gamma$, then $\mathcal
E(\eta(\b{\kappa}))=\b{\kappa}$, $\gamma(\eta(\b{\kappa}))=\gamma$,
as well as $|\,\eta\,|=\kappa\ell{+}h{+}(n{-}1)(\ell{-}1)=s$, i.e.,
$\eta\in\mathbb Z_+^n(s,\bold m)$. So,
$\wp=\{\,\eta(\b{\kappa})\,{=}\,\bigl(\kappa_1\ell{+}(\ell{-}1),\cdots,\kappa_{n-1}\ell{+}(\ell{-}1),\kappa_n\ell{+}h\bigr)\mid
\sum\kappa_i=\kappa$, $0\le \kappa_i\le m{-}1\}$.

According to Proposition 3.5 and the above conclusion (1), we see
that $x^{(\eta(\b{\kappa}))}$ is the highest weight vector of the
simple module $\mathfrak V_{\eta(\b{\kappa})}$. As the $n$-tuples in
$\wp$ are not equivalent with each other with respect to
$\sim$, from the proof of the above conclusion (2), we obtain that
$\mathrm{Soc}\,\mathcal{A}_q^{(s)}(n, \bold
m)=\bigoplus_{\eta(\b{\kappa})\in \wp}\mathfrak
V_{\eta(\b{\kappa})}$ is non-simple.

Now we claim that $\mathcal{A}_q^{(s)}(n, \bold m)$ is
indecomposable.

\smallskip
(I) Denote $\mathcal K(\kappa):=\{\,\bold
0\le\b{\kappa}=(\kappa_1,\cdots,\kappa_n)\le(m{-}1,\cdots,m{-}1)\mid
\sum\kappa_i=\kappa\,\}$. Now let us lexicographically order the
$n$-tuples in $\mathcal K(\kappa)$ as follows.
\begin{equation*}
\begin{split}
\b{\kappa}&\succ\b{\kappa}{-}\varepsilon_{n-1}{+}\varepsilon_n\\
&\succ\cdots\\
&\succ\b{\kappa}{-}\varepsilon_j{+}\varepsilon_{j+1}
\succ\b{\kappa}{-}\varepsilon_j{+}\varepsilon_{j+2}\succ
\cdots\succ\b{\kappa}{-}\varepsilon_j{+}\varepsilon_n\\
&\succ\cdots\\
&\succ\b{\kappa}{-}\varepsilon_i{+}\varepsilon_{i+1}\succ
\b{\kappa}{-}\varepsilon_i{+}\varepsilon_{i+2}
\succ\cdots\succ\b{\kappa}{-}\varepsilon_i{+}\varepsilon_n\quad(\,\text{\it $i$-th line appears if $\kappa_i>0$}\,)\\
&\succ\cdots\\
&\succ\b{\kappa}{-}\varepsilon_1{+}\varepsilon_2
\succ\b{\kappa}{-}\varepsilon_1{+}\varepsilon_3\succ\cdots\succ\b{\kappa}{-}\varepsilon_1{+}\varepsilon_n\succ
\cdots.
\end{split}
\end{equation*}
So $(\mathcal K(\kappa), \succ)$ is a totaly ordered set. Actually,
the lexicographic order $\succ$ on each line exactly coincides with
the pre-order $\succcurlyeq$ given by the type-$A$ weight system
(relative to its prime root system
$\{\varepsilon_i{-}\varepsilon_{i+1}\mid 1\le i<n\}$\,), i.e.,
$\b{\kappa}{+}\varepsilon_i{-}\varepsilon_{i+1}\succ\b{\kappa}\,,\Longrightarrow
\b{\kappa}{+}\varepsilon_i{-}\varepsilon_{i+1}\succcurlyeq\b{\kappa}\,$.
The latter pre-order will be used in dealing with the
$\mathfrak{u}_q(\mathfrak{sl}_n)$-action below. Now we suppose that
$(\mathcal K(\kappa{+}i), \succ)$ is totaly ordered for each $0\le
i\le E(s)-\kappa$.

(II) For any two successive $n$-tuples
$\bigl(\b{\kappa},\,\b{\kappa}'\bigr)$ in $\mathcal K(\kappa)$,
$\b{\kappa}\succ\b{\kappa}'$, either (i): $\b{\kappa}, \b{\kappa}'$
lies in the same line of some $\b{\kappa}''\in\mathcal K(\kappa)$,
as shown in the Figure above, then there exist $i<j\,(<n)$, such
that $\kappa_i''>0$,
$\b{\kappa}=\b{\kappa}''{-}\varepsilon_i{+}\varepsilon_j$ and
$\b{\kappa}'=\b{\kappa}''{-}\varepsilon_i{+}\varepsilon_{j+1}$, that
is,
$\b{\kappa}=\b{\kappa}'{+}\varepsilon_j{-}\varepsilon_{j+1}\succcurlyeq\b{\kappa}'$;
or (ii): $\b{\kappa}$ lies in the end of the $j$-th line of some
$\b{\kappa}''\in\mathcal K(\kappa)$, i.e., $\kappa_j''>0$,
$\b{\kappa}=\b{\kappa}''-\varepsilon_j{+}\varepsilon_n$, and
$\b{\kappa}'$ lies in the ahead of the $i$-th line with
$\kappa_t''=0$ for $i<t<j$ and $\kappa''_i>0$, i.e.,
$\b{\kappa}'=\b{\kappa}''-\varepsilon_i{+}\varepsilon_{i+1}$. Even
for the latter, $\b{\kappa}\not\succcurlyeq\b{\kappa}'$, but we have
$\b{\kappa}''=\b{\kappa}{+}\varepsilon_j{-}\varepsilon_n\succcurlyeq\b{\kappa}$
and
$\b{\kappa}''=\b{\kappa}'{+}\varepsilon_i{-}\varepsilon_{i+1}\succcurlyeq\b{\kappa}'$.

Now both cases reduce to treat the general case:
$\bigl(\b{\kappa}{+}\varepsilon_i{-}\varepsilon_j,\,
\b{\kappa}\bigr)$, with the pre-order
$\b{\kappa}{+}\varepsilon_i{-}\varepsilon_j\succcurlyeq\b{\kappa}$,
 where
$\b{\kappa}{+}\varepsilon_i{-}\varepsilon_j=(\kappa_1,\cdots,\kappa_i{+}1,\cdots,\kappa_j{-}1,\cdots,\kappa_n)$,
and $j \, ({>}\,i)$ is the first index such that $\kappa_j\ne0$. So,
there exists a $\b{\kappa}{+}\varepsilon_i\in\mathcal
K(\kappa{+}1)$, such that
$\b{\kappa}{+}\varepsilon_i\succeq\b{\kappa}{+}\varepsilon_i{-}\varepsilon_j$
and $\b{\kappa}{+}\varepsilon_i\succeq\b{\kappa}$ (Note the
pre-order $\succeq$ here defined as before in subsection 3.3).

To $\b{\kappa}{+}\varepsilon_i$, we can associate two equivalent
$n$-tuples: $\theta_i\sim\vartheta_i\in\mathbb Z_+^n(s, \bold m)$,
where
\begin{equation*}
\begin{split}
\theta_i&
=\bigl(\kappa_1\ell{+}(\ell{-}1),\cdots,(\kappa_i{+}1)\ell,\cdots,\kappa_j\ell{+}(\ell{-}2),\cdots,\kappa_n\ell{+}h\bigr),\\
\vartheta_i&=\bigl(\kappa_1\ell{+}(\ell{-}1),\cdots,(\kappa_i{+}1)\ell{+}(\ell{-}2),\cdots,
\kappa_j\ell,\cdots,\kappa_n\ell{+}h\bigr),
\end{split}
\end{equation*}
with $\mathcal E(\theta_i)=\mathcal
E(\vartheta_i)=\b{\kappa}{+}\varepsilon_i$. According to the
formulae (1) \& (2) in (\cite{HU}, 4.5) and Proposition 4.6 of
\cite{HU}, there are quantum root vectors $f_{\alpha_{ij}},
\,e_{\alpha_{ij}}\in\mathfrak{u}_q(\mathfrak{sl}_n)$ associated to
positive root $\alpha_{ij}=\varepsilon_i{-}\varepsilon_j$, such that
$ f_{\alpha_{ij}}.\,x^{(\theta_i)}=c_1x^{(\eta(\b{\kappa}))}$, and
$e_{\alpha_{ij}}.\,x^{(\vartheta_i)}=c_2x^{(\eta(\b{\kappa}{+}\varepsilon_i{-}\varepsilon_j))}$,
$(c_1,\,c_2\in k^*)$. However,
$\mathfrak{u}_q(\mathfrak{sl}_n).\,x^{(\theta_i)}=
\mathfrak{u}_q(\mathfrak{sl}_n).\,x^{(\vartheta_i)}$, that is,
$\mathfrak V_{\eta(\b{\kappa})}\bigoplus\mathfrak
V_{\eta(\b{\kappa}{+}\varepsilon_i{-}\varepsilon_j)}\subsetneq
\mathfrak V_{\theta_i}=\mathfrak V_{\vartheta_i}$.

In summary, for any two successive $n$-tuples $(\b{\kappa},
\b{\kappa}')$ in $\mathcal K(\kappa)$ with
$\b{\kappa}\succ\b{\kappa}'$, either $\mathfrak
V_{\eta(\b{\kappa})}\bigoplus\mathfrak V_{\eta(\b{\kappa}')}$ for
$\b{\kappa}\succcurlyeq\b{\kappa}'$, or $\mathfrak
V_{\eta(\b{\kappa}'')}\bigoplus\mathfrak
V_{\eta(\b{\kappa}')}\bigoplus\mathfrak V_{\eta(\b{\kappa})}$ for
$\b{\kappa}''\succcurlyeq\b{\kappa}'$ and
$\b{\kappa}''\succcurlyeq\b{\kappa}$, can be embedded into a larger
highest weight submodule  generated by highest weight vector
$x^{(\eta(\b{\kappa}'{+}\varepsilon_j))}$, or the
 sum of two larger highest weight submodules by highest
weight vectors $x^{(\eta(\b{\kappa}{+}\varepsilon_j))}$ and
$x^{(\eta(\b{\kappa}'{+}\varepsilon_i))}$, all lying in a higher
energy degree $\kappa{+}1$. Because $(\mathcal K(\kappa), \succ)$ is
totally ordered, taking over all the two successive $n$-tuples pairs
$(\b{\kappa}_i, \b{\kappa}_{i+1})$, for $i=1,2,\cdots,\#\mathcal
K(\kappa)$, we prove that $\mathrm{Soc}\,\mathcal{A}_q^{(s)}(n,
\bold m)=\bigoplus_{\eta(\b{\kappa})\in \wp}\mathfrak
V_{\eta(\b{\kappa})}$ can be pairwise intertwinedly embedded into
the sum of larger indecomposable highest weight submodules with
generators lying in a higher energy degree $\kappa{+}1$.

(III) Finally, note that each $(\mathcal K(\kappa{+}\imath),\succ)$
is totally ordered, for every $\imath=0,1,\cdots,E(s){-}\kappa$.
Repeating the proof for $\mathcal K(\kappa)$ in (II), we can lift
pairwise intertwinedly the sum of highest weight submodules at each
energy level into the sum of larger highest weight submodules with
highest weight vectors lying in a higher one level, up to the top
energy level $E(s)$, so that $\mathcal{A}_q^{(s)}(n, \bold m)$ is
indecomposable.

We complete the proof.
\end{proof}
%\[
%\xymatrix{&x^{(\theta_1)} \ar[dr]^{v_1}\ar[dl]_{v'_1} && x^{(\theta_2)}\ar[dl]_{v'_2}&\cdots
%& x^{(\theta_j)}\ar[dr]^{v'_j}&&x^{(\theta_{j+1})}\ar[dl]_{v_{j+1}}\ar[dr]^{v'_{j+1}} \\
%   x^{(\beta)} && x^{(\delta_1)}&&& &x^{(\delta_{j})}&&x^{(\gamma)} ~~~~.   }
%\]

\begin{remark} We develop a new ``intertwinedly-lifting"
method to prove the
indecomposability of $\mathcal{A}^{(s)}_q(n, \bold m)$ in the case when its socle submodule is non-simple.
Note that the indecomposability of $\mathcal{A}_q^{(s)}(n, \bold
m)$ when its socle is non-simple depends on
our assumption $n>2$. Its argument is subtle and more interesting. An
intrinsic reason for resulting in the indecomposability in this case
is revealed by the existing difference between $E(s)_0$ and $E(s)$ as
depicted in our result, see Lemma 3.7 (3), occurred only under the above
assumption. Although our module model $\mathcal{A}_q^{(s)}(n, \bold
m)$ is still valid to the analysis of the submodule structures in
the rank $1$ case, namely, for $\mathfrak{u}_q(\mathfrak{sl}_2)$,
there exists an essential difference between our case here
$\mathfrak{u}_q(\mathfrak{sl}_n)$ with $n>2$ and
$\mathfrak{u}_q(\mathfrak{sl}_2)$. While, the indecomposable modules
for the latter has been completely solved in different perspectives
by many authors, like Chari-Premet \cite{CP}, Suter \cite{Su},
Xiao \cite{XJ}, etc. Recently, for the even order of root of
unity case, Semikhatov \cite{S} distinctly analyzed the submodules
structure of the divided-power quantum plane for the Lusztig small quantum
group $\bar{\mathfrak{u}}_q(\mathfrak{sl}_2)$ using a different way.
\end{remark}

\noindent {\it 3.5. Loewy filtration of $\mathcal{A}_q^{(s)}(n, \bold
m)$ and Loewy layers.} As shown in Theorem 3.8 (5), for the given
$s$ with $\ell\le s\le N{-}\ell$, $\mathcal{A}_q^{(s)}(n, \bold m)$
is indecomposable. We will adopt a method of the filtration analysis
to explore the submodule structures for the indecomposable
module $\mathcal{A}_q^{(s)}(n, \bold m)$.

Set
 $\mathcal V_0=\mathrm{Soc}\,\mathcal{A}_q^{(s)}(n,
\bold m)$, and for $i>0$,
$$
\mathcal
V_i=\span_k\left.\Bigl\{\,x^{(\alpha)}\in\mathcal{A}_q^{(s)}(n,
\bold m)\,\right|\,E(s)_0\leq \mathrm{Edeg}\,x^{(\alpha)}\leq
E(s)_0{+}i\, \Bigr\}.
$$
Obviously, $\mathcal V_{i-1}\subseteq \mathcal V_i$, for any $i$.

Denote $\mathcal K_i^{(s)}:=\mathcal
K\bigl(E(s)_0{+}i\bigr)\,{=}\,\{\,\b{\kappa}=(\kappa_1,\cdots,\kappa_n)\mid
|\,\b{\kappa}\,|=E(s)_0{+}i, \ \kappa_i\le m{-}1\}$, for $0\le i\le
E(s){-}E(s)_0$.

Set $\eta_i=(\ell{-}1, \cdots, \ell{-}1, \overset{t_i}{h_i}, 0,
\cdots, 0)$ and $s_i=|\,\eta_i\,|=(t_i{-}1)(\ell{-}1)+h_i$, for
$1\le t_i\le n$ and $0\leq h_i\leq \ell{-}1$. Write
$\eta(\b{\kappa},i):=\ell\,{\cdot}\,\b{\kappa}+\eta_i$, such that
$|\,\eta(\b{\kappa},i)\,|=s$. Set
$\wp_i^{(s)}:=\{\,\eta(\b{\kappa},i)\in\mathbb{Z}^n_+(s,\bold m)\mid
s=\bigl(E(s)_0{+}i\bigr)\ell+s_i\,\}$. Particularly, for
$n(\ell{-}1){+}1\le s\le N{-}\ell$, $\wp_0^{(s)}=\wp$, as defined in
Theorem 3.8. Note that for any $\ell\le s\le N{-}\ell$, one has
$t_i<n$, for $i>0$.

\begin{theorem} Suppose $n\geq3$ and $\mathbf{char}(q)=\ell\ge
3$. For the indecomposable $\mathfrak{u}_q(\mathfrak{sl}_n)$-modules
$\mathcal{A}_q^{(s)}(n, \bold m)$ with $\ell\le s\le N{-}\ell$, one
has

$(1)$ $\mathcal V_i$'s are
$\mathfrak{u}_q(\mathfrak{sl}_n)$-submodules of
$\mathcal{A}_q^{(s)}(n, \bold m)$, and the filtration
$$
0\subset \mathcal V_0\subset \mathcal V_1\subset\cdots\subset
\mathcal V_{E(s)-E(s)_0}=\mathcal{A}_q^{(s)}(n, \bold
m)\leqno(\divideontimes)
$$
is a $\mathrm{Loewy \ filtration}$ of $\mathcal{A}_q^{(s)}(n, \bold
m)$.

$(2)$ $x^{(\eta(\b{\kappa},i))}\in\mathcal{A}_q^{(s)}(n, \bold m)$
are primitive vectors of  $\mathcal V_i$ $($relative to $\mathcal
V_{i-1}$$)$, for all $\b{\kappa}\in \mathcal K_i^{(s)}$, and
$\mathfrak{u}_q(\mathfrak{sl}_n).\,(x^{(\eta(\b{\kappa},i))}+
\mathcal V_{i-1})\cong
\mathfrak{u}_q(\mathfrak{sl}_n).\,x^{(\eta_i)}=\mathfrak
V_{\eta_i}$.  Its $i$-th Loewy layer
\[
\begin{split}
\mathcal V_{i}/{\mathcal V_{i-1}}&=\span_k\{\,x^{(\alpha)}+\mathcal
V_{i-1}\mid\mathrm{Edeg}\,x^{(\alpha)}=E(s)_0+i\,\}\\
&=
\bigoplus_{\eta(\b{\kappa},i)\in\wp_i^{(s)}}\mathfrak{u}_q(\mathfrak{sl}_n).\,(x^{(\eta(\b{\kappa},i))}+
\mathcal V_{i-1})\\
&\cong\bigl(\#\mathcal K_i^{(s)}\bigr)\,\mathfrak
V_{\eta_i} 
\end{split}
\]
is the direct sum of $\#\mathcal K_i^{(s)}$ isomorphic 
copies of simple module $\mathfrak V_{\eta_i}=\mathcal{A}_q^{(s_i)}(n, \bold 1)$.
\end{theorem}
\begin{proof} By definition of $E(s)_0$,
 $\mathrm{Edeg}\,(u.\, x^{(\alpha)})\geq E(s)_0$, only if $u.\, x^{(\alpha)}\ne 0$, for any $0\neq u\in
\mathfrak{u}_q(\mathfrak{sl}_n)$, $x^{(\alpha)}\in \mathcal V_i$.
Meanwhile, Proposition 3.3 gives rise to $\mathrm{Edeg}\,(u.\,
x^{(\alpha)})\leq \mathrm{Edeg}\, x^{(\alpha)}\leq E(s)_0+i$. Thus,
Definition 3.2 implies that $\mathcal V_i$ is a
$\mathfrak{u}_q(\mathfrak{sl}_n)$-submodule of
$\mathcal{A}_q^{(s)}(n, \bold m)$. So we get a filtration
$(\divideontimes)$ of submodules of $\mathcal{A}_q^{(s)}(n, \bold
m)$.

On the other hand, if $\mathrm{Edeg}\,x^{(\alpha)}=E(s)_0\,{+}\,i$,
then $x^{(\alpha)}\notin \mathcal V_{i-1}$, by definition, $\mathcal
 V_i/\mathcal V_{i-1}$ is spanned by $\left\{\,x^{(\alpha)}+\mathcal
V_{i-1}\,\big|\,\mathrm{Edeg}\,
x^{(\alpha)}=E(s)_0\,{+}\,i\,\right\}$.

Assert that $x^{(\eta(\b{\kappa},i))}$ is a primitive vector of
$\mathcal V_i$ relative to $\mathcal V_{i-1}$ $(i\geq 1)$. In fact,
\[
e_j.\,x^{(\eta(\b{\kappa},i))}=\left\{\begin{aligned}
&[\,\kappa_j\ell{+}\ell\,]\,x^{(\eta(\b{\kappa},i)+\varepsilon_j-\varepsilon_{j+1})}=0,~~&
j<t_i, \\
&[\,\kappa_j\ell{+}\delta_{j,t_i}h_i{+}1\,]\,x^{(\eta(\b{\kappa},i)+\varepsilon_j-\varepsilon_{j+1})},~~&j\ge
t_i.
\end{aligned}
\right.
\]

(i) When $t_i=n$: since $e_j.\,x^{(\eta(\b{\kappa},i))}=0$ for $1\le
j<n$, $x^{(\eta(\b{\kappa},i))}$ is a maximal weight vector.

(ii) When $t_i<n$: either
$e_j.\,x^{(\eta(\b{\kappa},i))}=0\in\mathcal V_{i-1}$ for $j<t_i$,
or
$e_j.\,x^{(\eta(\b{\kappa},i))}=c\,x^{(\eta(\b{\kappa},i)+\varepsilon_j-\varepsilon_{j+1})}\in\mathcal
V_{i-1}$ for $j\ge t_i$ and $c\in k^*$. So, $x^{(\eta(\b{\kappa},i))}$ is a
primitive vector of $\mathcal V_i$ relative to $\mathcal V_{i-1}$
$(i\geq 1)$.

Set $\overline{\mathcal
V}_{\eta(\b{\kappa},i)}:=\mathfrak{u}_q(\mathfrak{sl}_n).\,(x^{(\eta(\b{\kappa},i))}+\mathcal
V_{i-1})$. By Proposition 3.5 \& Theorem 2.5 (2), we get that
\[
\begin{split}
\overline{\mathcal V}_{\eta(\b{\kappa},i)}&\cong
\mathfrak{u}_q(\mathfrak{sl}_n).\,x^{(\eta(\b{\kappa},i))}/\bigl(\mathfrak{u}_q(\mathfrak{sl}_n).\,x^{(\eta(\b{\kappa},i))}\cap
\mathcal V_{i-1})\\
&\cong \mathfrak{u}_q(\mathfrak{sl}_n).\,x^{(\eta_i)}=\mathfrak
V_{\eta_i}=\mathcal{A}_q^{(s_i)}(n, \bold 1).
\end{split}
\]
So, $\overline{\mathcal V}_{\eta(\b{\kappa},i)}$ is a simple
submodule of ${\mathcal V_i}/{\mathcal V_{i-1}}$.

For any $\b{\kappa}, \ \b{\kappa}'\in\mathcal K_i^{(s)}$ with
$\b{\kappa}\ne\b{\kappa}'$, i.e.,
$\eta(\b{\kappa},i)\nsim\eta(\b{\kappa}',i)$,  by Proposition 3.4,
$\overline{\mathcal V}_{\eta(\b{\kappa},i)}$, $\overline{\mathcal
V}_{\eta(\b{\kappa}',i)}$ are simple submodules of $\mathcal
V_i/\mathcal V_{i-1}$ with $\overline{\mathcal
V}_{\eta(\b{\kappa},i)}\cap\overline{\mathcal
V}_{\eta(\b{\kappa}',i)}=0$, but $\overline{\mathcal
V}_{\eta(\b{\kappa},i)}\cong\overline{\mathcal
V}_{\eta(\b{\kappa}',i)}\cong \mathfrak
V_{\eta_i}=\mathcal{A}_q^{(s_i)}(n, \bold 1)$. As $\wp_i^{(s)}$
parameterizes the generator set of $\mathcal V_i/\mathcal V_{i-1}$,
$\mathcal V_i/\mathcal V_{i-1}=\bigoplus_{\b{\kappa}\in\mathcal
K_i^{(s)}}\overline{\mathcal
V}_{\eta(\b{\kappa},i)}\cong\bigl(\#\mathcal
K_i^{(s)}\bigr)\,\mathfrak V_{\eta_i}$.

As shown in the proof of Theorem 3.8 (5), $\mathcal V_i/\mathcal
V_{i-2}$ is indecomposable for any $i$ \
$(\,1\,{<}\,i\,{\leq}\,E(s)\,{-}\,E(s)_0\,)$. Hence, the filtration
$(\divideontimes)$ is not contractible and  has the shortest length
such that $\mathcal V_i/\mathcal V_{i-1}$ are semisimple, then it is
a Loewy filtration (for definition, see \cite{JEH}).
\end{proof}

As a consequence of Theorem 3.10, we obtain an interesting combinatorial identity
below.
\begin{coro} $(\text{\rm i})$ $\#\mathcal K_i^{(s)}
=\sum_{j=0}^{\lfloor
\frac{E(s)_0{+}i}{m}\rfloor}(-1)^j\binom{n}{j}\binom{n{+}(E(s)_0{+}i){-}jm{-}1}{n{-}1}$.

$(\text{\rm ii})$ $\sum_{i=0}^{\lfloor
\frac{s}{m\ell}\rfloor}(-1)^i\binom{n}{i}\binom{n{+}s{-}im\ell{-}1}{n{-}1}=\sum_{i=0}^{E(s){-}E(s)_0}\Bigl(\sum_{j=0}^{\lfloor
\frac{s_i}{\ell}\rfloor}(-1)^j\binom{n}{j}\binom{n{+}s_i{-}j\ell{-}1}{n{-}1}\Bigr)\times
\\
\times\Bigl(\sum_{j=0}^{\lfloor
\frac{E(s)_0{+}i}{m}\rfloor}(-1)^j\binom{n}{j}\binom{n{+}(E(s)_0{+}i){-}jm{-}1}{n{-}1}\Bigr)$,
where $s=(E(s)_0{+}i)\ell+s_i$.
\end{coro}
\begin{proof} (i) From the definition of $\mathcal K_i^{(s)}$, $\#\mathcal K_i^{(s)}$
is equal to the coefficient of $t^{E(s)_0+i}$ of polynomial $P_{n,
m}(t)=(1+t+t^2+\cdots+t^{m-1})^n$. So, it is true, similar to
Corollary 2.6.

(ii) follows from (i), Proposition 3.1 \& Corollary 2.6, as well as
\[
\mathcal{A}_q^{(s)}(n, \bold
m)\cong\bigoplus_{i=0}^{E(s){-}E(s)_0}\mathcal V_i/\mathcal
V_{i-1}\cong\bigoplus_{i=0}^{E(s){-}E(s)_0}\bigl(\#\mathcal
K_i^{(s)}\bigr)\mathcal{A}_q^{(s_i)}(n, \bold 1),\leqno(\circledast)
\]
as vector spaces.
\end{proof}

%By Propositions 2.11--2.13, we achieve the transferring
%relationships among the elements in different Loewy layer over the
%actions of $\mathfrak{u}_q(\mathfrak {sl}_n)$, thus the structure of
%the whole indecomposable module is very clear.

%By Proposition 2.12, each simple factor module of each Loewy layer
%is corresponding to unique equivalent class in Remark 1, and
%furthermore, each simple module is generated by the image of
%elements in corresponding equivalent class.
Now we give an example to show the structural variations of
$\mathcal{A}_q^{(s)}(n, \bold m)$ by increasing the degree $s$. For
$n=3, m=2$ and $\ell=3$, in the following picture, each point
represents one simple submodule of a Loewy layer, and each arrow
represents the linked relationships existed among the simple
subquotients. For example, $a\rightarrow b$ means that there exists
$u\in \mathfrak{u}_q(\mathfrak {sl}_3)$ such that $u.\, a=b$.
\\
\\
\\
\\
\\
\\
\\
\\
\\
\\
\\
\\
\\
\\
\\
\\
\\
\\
\scalebox{.6}[.5]{\includegraphics[0, -6][10, 0]{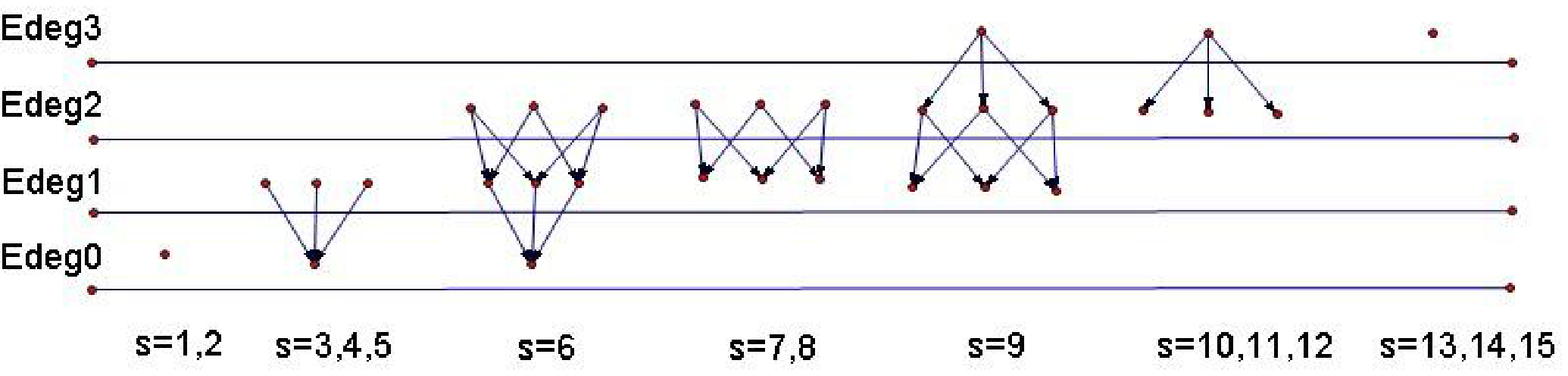}}

\noindent {\it 3.6. Rigidity of $\mathcal{A}_q^{(s)}(n, \bold m)$.}
As we known, both the radical filtration and the socle filtration of
a module $M$ are the Loewy filtrations, and $\text{\rm
Rad}^{r-k}M\subseteq \text{\rm Soc}^kM$, where $r=\ell\ell\,M$ is
the Loewy length of $M$. In this subsection, we will prove the
coincidence of both filtrations for $\mathcal{A}_q^{(s)}(n, \bold
m)$, that is the following result.
\begin{theorem}
Suppose $n\geq3$ and $\mathbf{char}(q)=\ell\ge 3$. Then
$\mathcal{A}_q^{(s)}(n, \bold m)$ is a $\footnote{The definitions of
rigid module, socle filtration, radical filtration can be found in
(\cite{JEH}, 8.14), or some relevant elegant investigations on rigidity of a module and Loewy filtration in \cite{I1, I2}.\ } \mathrm{rigid}$
$\mathfrak{u}_q(\mathfrak{sl}_n)$-module, and $\ell\ell\,\mathcal{A}^{(s)}_q(n, \bold m)=E(s){-}E(s)_0{+}1$.
\end{theorem}
\begin{proof} By the definition of rigid module, it suffices to
prove that the filtration $(\divideontimes)$ in Theorem 3.10 is both
socle and radical.

(1) Note $\mathrm{Soc}^0\mathcal{A}_q^{(s)}(n, \bold m)=0$,
$\mathrm{Soc}^1\mathcal{A}_q^{(s)}(n, \bold m)=\mathcal V_0$, by
Theorem 3.10. Assume that we have proved
$\mathrm{Soc}^i\mathcal{A}_q^{(s)}(n, \bold m)=\mathcal V_{i-1}$,
for $i\ge 1$. We are going to show $\mathcal V_i/\mathcal
V_{i-1}=\mathrm{Soc}\,(\mathcal{A}^{(s)}_q(n, \bold m)/\mathcal
V_{i-1})$, i.e., $\mathrm{Soc}^{i+1}\mathcal{A}_q^{(s)}(n, \bold
m)=\mathcal V_i$.

As $\mathcal V_i/\mathcal
V_{i-1}=\bigoplus_{\eta(\b{\kappa},i)\in\wp_i^{(s)}}\overline{\mathcal
V}_{\eta(\b{\kappa},i)}\bigl(\subset \mathcal{A}^{(s)}_q(n, \bold
m)/\mathcal V_{i-1}\bigr)$ is semisimple, $\mathcal V_i/\mathcal
V_{i-1}\subseteq \mathrm{Soc}\,\bigl(\mathcal{A}^{(s)}_q(n, \bold
m)/\mathcal V_{i-1}\bigr)$. Note that $\mathcal V_i$ is spanned by
$\{\,x^{(\alpha)}\in \mathcal{A}^{(s)}_q(n, \bold
m)\mid\mathrm{Edeg}\,x^{(\alpha)}\,{=}$ $E(s)_0\,{+}\,i\,\}$, for
each $i\ge 1$. Similarly to Theorem 3.8 (1), we assert that for any
nonzero $y{+}\mathcal V_{i-1}\in\mathcal{A}_q^{(s)}(n, \bold
m)/\mathcal V_{i-1}$ with energy degree $\mathrm{Edeg}\,(y)\ge
E(s)_0{+}i$, assume that the submodule $\overline{\mathfrak
V}_y=\mathfrak{u}_q(\mathfrak{sl}_n).\,(y{+}\mathcal V_{i-1})$ is
simple, then $\mathrm{Edeg}\,(y)=E(s)_0\,{+}\,i$, that is,
$y\in\mathcal V_i$. This gives the desired result.

In fact, if $\mathrm{Edeg}\,(y)>E(s)_0\,{+}\,i$, that is,
$\mathrm{Edeg}\,(y)=E(s)_0\,{+}\,j$ with $j>i$, then by Definition
3.2, in the expression of $y=\sum_{\alpha}k_\alpha x^{(\alpha)}$,
there exists some $\beta\in\mathbb Z_+^n(s,\bold m)$, $k_\beta\ne 0$
such that $|\,\mathcal E(\beta)\,|=\mathrm{Edeg}\,(y)$. Write
$\b{\kappa}=\mathcal E(\beta)$. Then there exists
$\eta(\b{\kappa},j)=\ell\,{\cdot}\,\b{\kappa}{+}\eta_j\in \mathbb
Z_+^n(s,\bold m)$ (where
$\eta_j=(\ell{-}1,\cdots,\ell{-}1,\underset{t_j}{h_j},0,\cdots,0)$
with $|\,\b{\kappa}\,|=E(s)_0{+}j\ge j$, (so $\exists\,
\kappa_{i_0}\ne 0)$, such that $\eta(\b{\kappa},j)\sim\beta$, by the
remark in subsection 3.3. Since $j>i\ge 1$, by the note previous to
Theorem 3.10, $t_j<n{-}1$. So, there is $\bar \eta_j=(\bar
h_1,\cdots,\bar h_n)$ with $\bar h_{i_0}=0$, $\bar
h_{i_0+1}<\ell{-}1$ and $\bar h_k\le \ell{-}1$ and
$|\,\bar\eta_j\,|=s_j=|\,\eta_j\,|$, such that $\bar \eta_j\sim
\eta_j$, and $\bar
\eta(\b{\kappa},j)=\ell\,{\cdot}\,\b{\kappa}{+}\bar\eta_j\sim\eta(\b{\kappa},j)\sim\beta$.
By Proposition 3.5, we can find
$u\in\mathfrak{u}_q(\mathfrak{sl}_n)$ such that
$u.\,x^{(\beta)}=x^{(\bar \eta(\b{\kappa},j))}$. Clearly, for
$x^{(\bar \eta(\b{\kappa},j))}$, there exists an $f_{i_0}\in
\mathfrak{u}_q(\mathfrak{sl}_n)$, such that $f_{i_0}.\,x^{(\bar
\eta(\b{\kappa},j))}\ne 0$ (then $(f_{i_0}u).\, y\neq 0$) but
$\mathrm{Edeg}\,(f_{i_0}.\,x^{(\bar
\eta(\b{\kappa},j))})=\mathrm{Edeg}\,x^{(\bar
\eta(\b{\kappa},j))}{-}1$, so
$\mathrm{Edeg}\,(f_{i_0}u.\,y)=\mathrm{Edeg}\,(f_{i_0}u.\,x^{(\beta)})=\mathrm{Edeg}\,(u.\,x^{(\beta)}){-}1
=\mathrm{Edeg}\,(x^{(\beta)}){-}1<\mathrm{Edeg}\,(y)$. Thereby, we
get a proper submodule $(0\ne)\, \overline{\mathfrak
V}_{f_{i_0}u.\,y}\subsetneq\overline{\mathfrak V}_{u.\,y}=
\overline{\mathfrak V}_y$, by Proposition 3.6. It is a
contradiction. So the above assertion is true.

%Because for any $x^{(\gamma)}\,{+}\,\mathcal V_{i-1}\in
%\mathcal{A}^{(s)}_q(n, \bold m)/\mathcal V_{i-1}$, there exists
%$u\in \mathfrak{u}_q(\mathfrak{sl}_n)$ such that $u.\,
%x^{(\gamma)}\in\mathcal  V_i$ but $u.\,x^{(\gamma)}\notin \mathcal
%V_{i-1}$. Set $W=\mathfrak{u}_q(\mathfrak{sl}_n).\,(u.\,
%x^{(\gamma)})\,{+}\,\mathcal V_{i-1}$, and it is a submodule of
%$\mathcal V_i/\mathcal V_{i-1}$, by Theorem 3.10. If
%$W'=\mathfrak{u}_q(\mathfrak{sl}_n).\,(x^{(\gamma)})+\mathcal
%V_{i-1}$ is a simple submodule of $\mathcal V_i/\mathcal V_{i-1}$,
%then it exists $v\in \mathfrak{u}_q(\mathfrak{sl}_n)$ such that
%$v\cdot (u\cdot x^{(\gamma)})=x^{(\gamma)}$. We have $\mathrm{Edeg}
%(u\cdot x^{(\gamma)})=\mathrm{Edeg}\, x^{(\gamma)}$, by Propositions
%2.6--2.8.
%Therefore, $\mathcal V_i/\mathcal V_{i-1}=
%\mathrm{Soc}\,(\mathcal{A}^{(s)}_q(n, \bold m)/\mathcal V_{i-1})$.
%That is, $(\divideontimes)$ is a socle filtration.

(2) By Theorem 3.10, $(\divideontimes)$ is a Loewy filtration of
$\mathcal{A}^{(s)}_q(n, \bold m)$, so its Loewy length
$r=\ell\ell\,\mathcal{A}^{(s)}_q(n, \bold m)=E(s){-}E(s)_0{+}1$.
%/\mathcal
%V_{E(s)-E(s)_0-1}=\bigoplus_{\b{\kappa}\in\mathcal
%K_{E(s)-E(s)_0}^{(s)}}\overline{\mathcal
%V}_{\eta(\b{\kappa},E(s){-}E(s)_0)}$. By Theorem 3.8,
%$\{\,x^{(\alpha)}\in\mathcal{A}^{(s)}_q(n, \bold
%m)\,\left|\,\mathrm{Edeg}x^{(\alpha)}=E(s)\right.\,\}$ generates
%$\mathcal{A}^{(s)}_q(n, \bold m)$.
%For any $x^{(\beta)}\in\mathcal{A}^{(s)}_q(n, \bold m)$ with
%$\mathrm{Edeg}\,x^{(\beta)}=E(s)$, consider the maximal submodule
%generated by
%$$
%\left\{x^{(\alpha)}\in\mathcal{A}^{(s)}_q(n, \bold
%m)\,\Big|\,\mathrm{Edeg}\,x^{(\alpha)}=E(s)\right\}-
%\left\{u.x^{(\beta)}\,\Big|\,u\in \mathfrak{u}_q(\mathfrak {sl}_n),
%\mathrm{Edeg}\,u. x^{(\beta)}=\mathrm{Edeg}\,x^{(\beta)}\right\}.
%$$
%The intersection of these maximal submodules is $\mathcal
%$V_{E(s)-E(s)_0-1}$, then
Then for $0\le i\le E(s){-}E(s)_0$, we have
$$
\mathrm{Rad}^i(\mathcal{A}^{(s)}_q(n, \bold
m))\subseteq\mathrm{Soc}^{r-i}(\mathcal{A}^{(s)}_q(n, \bold m))
=\mathcal V_{E(s)-E(s)_0-i}.
$$

For $i=1$: if there exists a $(0\ne\,)\,y\in \mathcal
V_{E(s)-E(s)_0-1}$, and
$y\notin\mathrm{Rad}^1(\mathcal{A}^{(s)}_q(n, \bold m))$, then by
definition, there is a maximal proper submodule $\mathcal V\subset
\mathcal{A}^{(s)}_q(n, \bold m)$ such that $y\notin\mathcal V$.
Since $\mathcal V$ is maximal,
$\mathfrak{u}_q(\mathfrak{sl}_n).\,y+\mathcal
V=\mathcal{A}^{(s)}_q(n, \bold
m)=\sum_{|\,\alpha\,|=E(s)}\mathfrak{u}_q(\mathfrak{sl}_n).\,x^{(\alpha)}$,
by Theorem 3.8 (3). However, $\mathrm{Edeg}\,u.\,y\leq
\mathrm{Edeg}\,y=E(s){-}1$, so we derive that $\{\,x^{(\alpha)}\in
\mathcal{A}^{(s)}_q(n, \bold
m)\,\left|\,\mathrm{Edeg}\,x^{(\alpha)}=E(s)\right.\,\}\subseteq
\mathcal V$. Therefore, $\mathcal V=\mathcal{A}^{(s)}_q(n, \bold
m)$, it is contrary to the above assumption. This means
$\mathrm{Rad}^1(\mathcal{A}^{(s)}_q(n, \bold m))=\mathcal
V_{E(s)-E(s)_0-1}$.

Assume we have proved that $\mathrm{Rad}^i(\mathcal{A}^{(s)}_q(n,
\bold m))=\mathcal V_{E(s)-E(s)_0-i}$,  for $i\ge 1$. Note that
$\mathrm{Rad}^{i+1}(\mathcal{A}^{(s)}_q(n, \bold
m))\subseteq\mathcal V_{E(s)-E(s)_0-i-1}\subset \mathcal
V_{E(s)-E(s)_0-i}=\mathrm{Rad}^i(\mathcal{A}^{(s)}_q(n, \bold m))$.
By definition, $\mathrm{Rad}^{i+1}(\mathcal{A}^{(s)}_q(n, \bold m))$
is the intersection of all maximal submodule of
$\mathrm{Rad}^i(\mathcal{A}^{(s)}_q(n, \bold m))$. According to
Theorem 3.10 (2), we have that $\mathcal V_{E(s){-}E(s)_0{-}i}$ is
spanned by $\{\, x^{(\alpha)}\in\mathcal{A}_q^{(s)}(n, \bold
m)\,\left|\, \mathrm{Edeg}\,x^{(\alpha)}=E(s){-}i\right.\}$. Using
the similar argument for $i=1$, we can derive
$\mathrm{Rad}^{i+1}(\mathcal{A}^{(s)}_q(n, \bold m))=\mathcal
V_{E(s)-E(s)_0-i-1}$.

Consequently, the filtration $(\divideontimes)$ is a radical
filtration.
\end{proof}

Denote by $\mathcal{A}_q^{(s)}(n)$ the $s$-th homogenous space of
$\mathcal{A}_q(n)$.

\begin{coro} Suppose that $n\geq3$ and $\mathbf{char}(q)=\ell\ge
3$. Then $\mathfrak{u}_q(\mathfrak{sl}_n)$-submodules
$\mathcal{A}_q^{(s)}(n)$ of $\mathcal{A}_q(n)$ are indecomposable
and rigid.
\end{coro}
\begin{proof} Since for any $s\in\mathbb{N}$, there is $m\in
\mathbb{N}$ such that $(m{-}1)\ell\leq s\leq m\ell{-}1$, then
$\mathcal{A}_q^{(s)}(n, \bold m)= \mathcal{A}_q^{(s)}(n)$. By
Theorems 3.8 and 3.12, $\mathcal{A}_q^{(s)}(n)$ is indecomposable
and rigid $\mathfrak{u}_q(\mathfrak{sl}_n)$-module.
\end{proof}

\section{Quantum Grassmann algebra and quantum de Rham cohomology}

\noindent{\it 4.1. $q$-differential over $\mathcal{A}_q(n)$.}
Denote by  $\wedge_q(n)=k\{dx_1, \ldots, dx_n\}/((dx_i)^2,
dx_jdx_i$ $+\,q^{-1}dx_idx_j, i<j)$, the quantum exterior algebra over $k$. Let
$\wedge_q(n)_{(s)}$ be the $s$-th homogeneous subspace of
$\wedge_q(n)$, as we know
$$
\wedge_q(n)_{(s)}=\span_k\{\,dx_{i_1}\wedge
dx_{i_2}\wedge\cdots\wedge dx_{i_s}\mid\, 1\le i_1<i_2<\cdots <i_s\le n\,\}.
$$
Identifying $\wedge_q(n)_{(1)}$ with the $\mathfrak{u}_q(\mathfrak{sl}_n)$-module
$V(\lambda_1)$ with highest weight vector $dx_1$, then $\wedge_q(n)\cong k[A^{0|n}]$,
as $\mathfrak{u}_q(\mathfrak{sl}_n)$-modules.

\begin{defi} Define a linear mapping
$d: \mathcal{A}_q(n)\longrightarrow \mathcal{A}_q(n)\otimes_k \wedge_q(n)_{(1)}$ as
$$
dx^{(\alpha)}=\sum_{i=1}^n\partial_i(x^{(\alpha)})\otimes dx_i=\sum_{i=1}^nq^{-\varepsilon_i*\alpha}x^{(\alpha-\varepsilon_i)}\otimes
dx_i, \quad \forall\, x^{(\alpha)}\in\mathcal{A}_q(n).
$$
Then $d$ is called the $q$-differential on $\mathcal{A}_q(n)$.
\end{defi}

\begin{prop} The $q$-differential $d$ is a $\mathfrak{u}_q(\mathfrak
{sl}_n)$-module homomorphism, that is, $d(u.\, x)=u.\, dx$, for $u\in \mathfrak{u}_q(\mathfrak
{sl}_n)$, $x\in\mathcal{A}_q(n)$,
provided that $\wedge_q(n)_{(1)}\cong V(\lambda_1)$ as $\mathfrak{u}_q(\mathfrak{sl}_n)$-module
with highest weight vector $dx_1$.
\end{prop}
\begin{proof} It suffices to consider the actions of generators of
$\mathfrak{u}_q(\mathfrak{sl}_n)$ on the basis elements $x^{(\beta)}$ of
$\mathcal{A}_q(n)$.

%It is known that the generators of $\mathfrak{u}_q(\mathfrak{sl}_n)$
%are $e_i, \, f_i, \, K_i^{\pm}$ $(i=1, \cdots,  n-1)$, and the
%monomial basis of $\mathcal{A}_q(n)$ is
%$\left\{\,x^{(\alpha)}\in\mathcal{A}_q(n)\mid\alpha\in\mathbb{Z}^n_+\,\right\}$.

%For any $x^{(\beta)}\in\mathcal{A}_q(n)$:

\smallskip
(1) For $e_i$ $(i=1, \cdots, n{-}1)$: On the one hand, noting that
\[\begin{split}
q^{-\varepsilon_j*(\beta+\varepsilon_i-\varepsilon_{i+1})}&=q^{-\varepsilon_j*\beta}, \qquad \text{\it for } j\le i, \text{\it or } j>i{+}1,\\
q^{-\varepsilon_{i+1}*(\beta+\varepsilon_i-\varepsilon_{i+1})}&=q^{-\varepsilon_{i+1}*\beta}q^{-1},\qquad \text{\it for } j=i{+}1,
\end{split}
\]
we have
\[
\begin{split}
d(e_i.\,
x^{(\beta)})&=d([\beta_i{+}1]\,x^{(\beta+\varepsilon_i-\varepsilon_{i+1})})\\
&=[\beta_i{+}1]\,\sum_{j=1}^nq^{-\varepsilon_j*(\beta+\varepsilon_i-\varepsilon_{i+1})}
x^{(\beta+\varepsilon_i-\varepsilon_{i+1}-\varepsilon_j)}\otimes
dx_j\\
&=[\beta_i{+}1]\,\Bigl(\sum_{j<i,{\it or } j>i+1}q^{-\varepsilon_j*\beta}x^{(\beta-\varepsilon_j+\varepsilon_i-\varepsilon_{i+1})}\otimes
dx_j\\
&\quad+\,q^{-\varepsilon_i*\beta}x^{(\beta-\varepsilon_{i+1})}\otimes
dx_i+q^{-\varepsilon_{i+1}*\beta-1}x^{(\beta+\varepsilon_i-2\varepsilon_{i+1})}\otimes
dx_{i+1}\Bigr).
\end{split}
\]

On the other hand, as $\Delta(e_i)=e_i\otimes K_i+1\otimes e_i$, we have
\[
e_i.\,(x^{(\beta-\varepsilon_j)}\otimes dx_j)=\left\{
\begin{aligned}
&[\beta_i{+}1]\,x^{(\beta-\varepsilon_j+\varepsilon_i-\varepsilon_{i+1})}\otimes
dx_j, & j<i, {\it or } \ j>i{+}1 \\
&[\beta_i]\,q\,x^{(\beta-\varepsilon_{i+1})}\otimes dx_i,  &j=i\\&x^{(\beta-\varepsilon_{i+1})}\otimes dx_i&j=i{+}1\\
&\quad+\,q^{-1}[\beta_i{+}1]\,x^{(\beta+\varepsilon_i-2\varepsilon_{i+1})}\otimes
dx_{i+1}.
\end{aligned}
\right.
\]
Observing that $q^{-\varepsilon_i*\beta}[\beta_i{+}1]=q^{-\varepsilon_i*\beta}[\beta_i]\,q+q^{-\varepsilon_{i+1}*\beta}$,
we finally obtain
\[
\begin{split}
 e_i.\,d(x^{(\beta)})&
 =e_i.\,\Bigl(\sum_{j=1}^nq^{-\varepsilon_j*\beta}x^{(\beta-\varepsilon_j)}\otimes dx_j\Bigr)
 =\sum_{j=1}^nq^{-\varepsilon_j*\beta}e_i.\,(
 x^{(\beta-\varepsilon_j)}\otimes dx_j)\\
 &=d(e_i.\,x^{(\beta)}).
\end{split}
\]

\smallskip
(2) Similarly, we can check that $d(f_i.\,x^{(\beta)})=f_i.\,(dx^{(\beta)})$, for $1\le i<n$.

\smallskip
(3) For $K_i$ $(i=1, \cdots, n{-}1)$:
\[\begin{split}
K_i.\,dx^{(\beta)}&=\sum_{j=1}^nq^{-\varepsilon_j*\beta}
K_i.\,x^{(\beta-\varepsilon_j)}\otimes K_i.\, dx_j\\
&=\sum_{j=1}^nq^{-\varepsilon_j*\beta}q^{\beta_i-\delta_{ij}-\beta_{i+1}+\delta_{i+1,j}}x^{(\beta-\varepsilon_j)}\otimes
q^{\delta_{ij}-\delta_{i+1,j}}dx_j\\
&=d(K_i.\,x^{(\beta)}).
\end{split}\]

\smallskip
This completes the proof.
\end{proof}

\noindent
{\it 4.2. Quantum Grassmann algebra and quantum de Rham Complex.}
It is a well-known fact that there exists a braiding $\aleph: \wedge_q(n)_{(1)}\otimes\mathcal{A}_q(n)\longrightarrow \mathcal{A}_q(n)\otimes\wedge_q(n)_{(1)}$, which is a $\mathfrak{u}_q(\mathfrak{sl}_n)$-module homomorphism.
This $\aleph$ also induces braidings $\aleph_s: \wedge_q(n)_{(s)}\otimes\mathcal{A}_q(n)\longrightarrow \mathcal{A}_q(n)\otimes\wedge_q(n)_{(s)}$. Now let us define the quantum Grassmann algebra as follows.
\begin{defi}
Let $\Omega_q(n):=\mathcal{A}_q(n)\otimes\wedge_q(n)$ with product
$$(x^{(\alpha)}\otimes \omega_s)\cdot (x^{(\beta)}\otimes \omega_r)=
x^{(\alpha)}\aleph_s(\omega_s\otimes x^{(\beta)})\,\omega_r, \quad \omega_s\in\wedge_q(n)_{(s)}, \  \omega_r\in\wedge_q(n)_{(r)}.$$
$\Omega_q(n)$ is said the {\it quantum Grassmann algebra} over $\mathcal A_q(n)$.
 $\Omega_q(n)=\bigoplus_{s=0}^n\Omega_q(n)^{(s)}$, where $\Omega_q(n)^{(s)}:=\mathcal{A}_q(n)\otimes\wedge_q(n)_{(s)}$.
\end{defi}
Define the linear mappings as follows.
\begin{gather*}
d^s: \ \Omega_q(n)^{(s)}
\longrightarrow\Omega_q(n)^{(s+1)}, \\
d^s(x^{(\alpha)}\otimes dx_{i_1}\wedge\cdots\wedge
dx_{i_s})=\sum^n_{j=1}q^{-\varepsilon_j*\alpha}x^{(\alpha-\varepsilon_j)}\otimes
dx_j\wedge dx_{i_1}\wedge\cdots\wedge dx_{i_s}.
\end{gather*}
Specially,  $d^0=d$ for $s=0$; $d^n=0$ for $s=n$.

By Proposition 4.1, $d$ is a homomorphism of
$\mathfrak{u}_q(\mathfrak{sl}_n)$-modules, then by definition, it follows readily
that $d^s$ $(s=1, \cdots, n)$ are homomorphisms of
$\mathfrak{u}_q(\mathfrak{sl}_n)$-modules.

\begin{prop} $(\Omega_q(n), d^\bullet)$ is a complex, i.e.,
$d^{s+1}d^{s}=0$, for $s=0, 1, \cdots, n$.
\end{prop}
\begin{proof} Observing the relationships between $d^s$ and $d^0$ for
$s=1,\cdots,n$, it is enough to check the case $s=0$.

Consider the actions of $d^1d^0$ over the basis elements of
$\mathcal{A}_q(n)$.

For any $x^{(\beta)}\in\mathcal{A}_q(n)$,
\[
\begin{split}
d^1d^0(x^{(\beta)})&=
d^1\Bigl(\sum_{j=1}^nq^{-\varepsilon_j*\beta}x^{(\beta-\varepsilon_j)}\otimes
dx_j\Bigr)
\\&=\sum_{j=1}^n\sum_{i=1}^nq^{-\varepsilon_j*\beta}
q^{-\varepsilon_i*(\beta-\varepsilon_j)}x^{(\beta-\varepsilon_i-\varepsilon_j)}\otimes dx_i\wedge dx_j\\
&= \sum_{i<j}(q^{-\varepsilon_j*\beta-\varepsilon_i*(\beta-\varepsilon_j)}
-q^{-1-\varepsilon_i*\beta-\varepsilon_j*(\beta-\varepsilon_i)})\,x^{(\beta-\varepsilon_i-\varepsilon_j)}\otimes
dx_i\wedge dx_j\\
&=\sum_{i<j}q^{-\varepsilon_i*\beta-\varepsilon_j*\beta}(1-1)\,x^{(\beta-\varepsilon_i-\varepsilon_j)}\otimes
dx_i\wedge dx_j=0.
\end{split}
\]
Thus, $d^1d^0=0$. By definition, it is easy to see that
$$
d^{s+1}d^{s}=0,\qquad s=1, \cdots, n.
$$

This completes the proof.
\end{proof}

For the complex $(\Omega_q(n), d^s)$ given in Proposition 4.4, that is,
\begin{gather*}
0{\longrightarrow}\Omega_q(n)^{(0)} \stackrel{d^0}
{\longrightarrow}
\cdots\stackrel{d^{s-1}}
{\longrightarrow}\Omega_q(n)^{(s)}\stackrel{d^s}
{\longrightarrow}\Omega_q(n)^{(s+1)}\stackrel{d^{s+1}}
{\longrightarrow}\cdots
\stackrel{d^{n-1}}
{\longrightarrow}\Omega_q(n)^{(n)}\stackrel{d^{n}}
{\longrightarrow}0,
\end{gather*}
when $q=1$, this is the standard de Rham complex of polynomial algebra
with $n$ variables. Thus, we call it the {\it quantum de Rham complex}.

\smallskip
\noindent
{\it 4.3. Quantum de Rham subcomplex $(\Omega_q(n,\bold m), d^\bullet)$ and its cohomologies.}
Now define $\Omega_q(n,\bold m):=\bigoplus_{s=0}^n\Omega_q(n,\bold m)^{(s)}$, where
$\Omega_q(n,\bold m)^{(s)}=\mathcal{A}_q(n, \bold
m)\otimes\wedge_q(n)_{(s)}$. Note that $d^s(\Omega_q(n, \bold
m)^{(s)})\subseteq \Omega_q(n, \bold
m)^{(s+1)}$, for $s=0, 1, \cdots, n$.
So, we get a quantum de Rham subcomplex $(\Omega_q(n,\bold m), d^\bullet)$.

For $\gamma\in\mathbb{Z}^n_+$, denote briefly by $\Omega^{(s)}_\gamma$ the
weight space corresponding to the weight
$\gamma=\sum^n_{i=1}\gamma_i\varepsilon_i$ of $\Omega_q(n,
\bold m)^{(s)}$, then
$$
\Omega^{(s)}_\gamma=\span_k\left.\Bigl\{x^{(\gamma{-}\sum^s_{j=1}\varepsilon_{i_j})}{\otimes}
\,dx_{i_1}{\wedge\cdots\wedge}\,dx_{i_s}\in\Omega_q(n,
\bold m)^{(s)}\,\right|\,\bold 0\,{\le}\,\gamma{-}\sum^s_{j=1}\varepsilon_{i_j}\,{\le}\,\bold
m\Bigr\}.
$$

\begin{lemm}
Given $\gamma\in\mathbb{Z}^n_+$ with $k_\gamma$ coordinates equal to $m\ell$ and $h_\gamma$ coordinates
equal to $0$. Then $\Omega^{(s)}_\gamma\neq 0$ if and only if $k_\gamma\leq s$,
and $\dim\,\Omega^{(s)}_\gamma=\binom{n-k_\gamma-h_\gamma}{s-k_\gamma}$.
\end{lemm}
\begin{proof} For the given $\gamma\in\mathbb{Z}^n_+$ with
$\gamma_{i_1}=\cdots=\gamma_{{i_k}_\gamma}=m\ell$ and $\gamma_{\imath_1}=\cdots=\gamma_{{\imath_h}_\gamma}=0$, if
$\Omega^{(s)}_\gamma\neq0$, there exists a pairwise distinct sequence $(j_1,\cdots,j_s)$, such that
$\{j_1,\cdots,j_s\}\cap\{\imath_1,\cdots,{\imath_h}_\gamma\}=\varnothing$ and $0\neq
x^{(\gamma-\sum^s_{r=1}\varepsilon_{j_r})}\otimes dx_{j_1}\wedge \cdots
\wedge dx_{j_s}\in \Omega^{(s)}_\gamma$, then the pairwise distinct sequence $(i_1,\cdots,{i_k}_\gamma)$ is a subsequence
of $(j_1,\cdots,j_s)$, so,
$k_\gamma\leq s$. And vice versa.
Hence, $\dim\,\Omega^{(s)}_\gamma=\binom{n-k_\gamma-h_\gamma}{s-k_\gamma}$.
\end{proof}

\begin{theorem}
For the quantum de Rham subcomplex
$(\Omega_q(n, \bold m), d^\bullet)$ below,
$$
0{\longrightarrow}\Omega_q(n, \bold m)^{(0)}
\cdots\stackrel{d^{s-1}} {\longrightarrow}\Omega_q(n, \bold m)^{(s)}\stackrel{d^s}
{\longrightarrow}\Omega_q(n, \bold m)^{(s+1)}\stackrel{d^{s+1}}
{\longrightarrow}\cdots
\Omega_q(n, \bold m)^{(n)}\stackrel{d^{n}} {\longrightarrow}0,
$$
one has
\[
\begin{split}
H^s(\Omega_q(n, \bold m))&=\ker d^s/\im d^{s-1}\\
&\cong\bigoplus_{1\le i_1<\cdots<i_s\le n}k\,[\,x^{(\sum^s_{j=1}(m\ell-1)\varepsilon_{i_j})}\otimes dx_{i_1}\wedge
\cdots\wedge dx_{i_s}],
\end{split}
\]
as $k$-vector spaces, and $\dim H^s(\Omega_q(n, \bold m))=\binom{n}{s}$, for $s=0, 1,
\cdots, n$.
\end{theorem}
\begin{proof}
Note the facts that $\Omega_q(n,\bold m)^{(s)}=\bigoplus_{\gamma\in\mathbb Z_+^n}\Omega^{(s)}_\gamma$ and
each differential $d^s$ preserves the weight-gradings.
It suffices to consider the restriction of the complex to weight $\gamma$, for any given $\gamma$.
$$
0\longrightarrow\Omega_\gamma^{(0)}\stackrel{\mathrm{d^0}}\longrightarrow\cdots\stackrel{\mathrm{d^{s-1}}} \longrightarrow\Omega^{(s)}_\gamma\stackrel{d^s}
\longrightarrow\Omega^{(s+1)}_\gamma\stackrel{d^{s+1}}
\longrightarrow\cdots\stackrel{d^{n-1}}
\longrightarrow\Omega^{(n)}_\gamma\stackrel{d^n}
\longrightarrow0.
$$

If $\gamma$ has $k_\gamma$ coordinates equal to $m\ell$ and
$h_\gamma$ coordinates equal to $0$, then by Lemma 4.5,
$\dim\,\Omega^{(s)}_\gamma=\binom{n-k_\gamma-h_\gamma}{s-k_\gamma}$.

(1) Consider the action of $d^0$ on $\mathcal{A}_q(n, \bold m)$.

For $\gamma=\bold 0$, it is clear that
$d^0x^{(\gamma)}=dx^{(\gamma)}=0$.

For $\gamma\ne\bold 0$, there exists $\gamma_j\neq0$ for some $j\ (1\leq
j\leq n)$, then
$$
d^0x^{(\gamma)}=dx^{(\gamma)}=\sum^n_{i=1}q^{-\varepsilon_i*\gamma}x^{\gamma-\varepsilon_i}\otimes
dx_i\neq0.
$$

So, $\ker d^0\cong k$,
$H^0(\Omega_q(n, \bold m))=\ker d^0/\im d^{-1}\cong k$,
$\dim H^0(\Omega_q(n, \bold m))=1$.

(2) Consider the behavior of $d^{s-1}, d^s$ at $\Omega_\gamma^{(s)}$.

For $\bold 0\neq\gamma\in\mathbb{Z}^n_{\geq0}$, by Lemma 4.5, $k_\gamma<s$ if and only if $\Omega^{(s-1)}_\gamma\neq0$;
$k_\gamma=s$ if and only if $\Omega^{(s-1)}_\gamma=0$ and
$\Omega^{(s)}_\gamma\neq0$; and $k_\gamma>s$ if and only if
$\Omega^{(s)}_\gamma=0\,(\,=\Omega^{(s-1)}_\gamma)$.

Obviously, when $k_\gamma>s$,
$\im d^{s-1}|_{\Omega^{(s-1)}_\gamma}=\ker\,d^s|_{\Omega^{(s)}_\gamma}=0$.
Namely, this case is no contribution to $H^s(\Omega_q(n,\bold m))$.
So, it suffices to consider the cases $k_\gamma\le s$.

We are now in a position to show the following assertions by induction on
$s\geq1$:

Case ($\textrm{i}$): When $\Omega^{(s{-}1)}_\gamma\ne 0$, i.e., $k_\gamma<s$, we must have
$$
\im d^{s{-}1}|_{\Omega^{(s{-}1)}_\gamma}=\ker d^s|_{\Omega^{(s)}_\gamma},\qquad \dim\im d^s|_{\Omega^{(s)}_\gamma}=\binom{n{-}k_\gamma{-}h_\gamma{-}1}{s{-}k_\gamma}.
$$
So, this case is also no contribution to $H^s(\Omega_q(n,\bold m))$.

Case ($\textrm{ii}$): When $\Omega^{(s-1)}_\gamma=0$ but $\Omega^{(s)}_\gamma\neq0$, i.e., $k_\gamma=s$, we must have
$$
{\Omega^{(s)}_\gamma}=\span_k \left\{
x^{(\sum^{s}_{j=1}(m\ell-1)\varepsilon_{i_j})}\otimes dx_{i_1}\wedge\cdots\wedge
dx_{i_{s}}\right\}=\ker d^s|_{\Omega^{(s)}_\gamma}.
$$

In summary, the above analysis leads to
\begin{equation*}
\begin{split}
H^s(\Omega_q(n, \bold m))&=\ker d^s/\im d^{s-1}=\bigoplus_{\gamma\in\mathbb Z_+^n}\ker d^s|_{\Omega^{(s)}_\gamma}/\im d^{s-1}|_{\Omega^{(s-1)}_\gamma}\\
&\cong\bigoplus_{1\le i_1<\cdots<i_s\le n}k\,[\,x^{(\sum^s_{j=1}(m\ell-1)\varepsilon_{i_j})}\otimes dx_{i_1}\wedge
\cdots\wedge dx_{i_s}],
\end{split}
\end{equation*}
and
$\dim H^s(\Omega_q(n, \bold m))=\binom{n}{s}$.

\medskip
\noindent
{\it Proofs of cases} (i) \& (ii):

\smallskip
For $s=1$: Assume that $\bold 0\ne \gamma\in\mathbb Z_+^n$, without loss of generality.

When $\Omega^{(0)}_\gamma\neq0$, i.e.,
$k_\gamma=0$: $\bold 0<\gamma\le \bold m$, $\dim\Omega^{(0)}_\gamma=1$. This means
$\Omega_\gamma^{(1)}\ne0$.
Now assume $0\ne\sum^n_{j=1}a_jx^{(\gamma-\varepsilon_j)}\otimes dx_j\in
\ker d^1$ with $a_j\in k$, i.e.,
$$
d^1\Bigl(\sum^n_{j=1}a_jx^{(\gamma-\varepsilon_j)}\otimes dx_j\Bigr)=\sum_{i<j}\bigl(a_jq^{-\varepsilon_i*\gamma}-
a_iq^{-\varepsilon_j*\gamma}\bigr)\,x^{(\gamma-\varepsilon_i-\varepsilon_j)}\otimes
dx_i\wedge dx_j=0,
$$
we obtain a system of equations with indeterminates $a_i$ $(i=1,
\cdots, n)$:
$$
a_jq^{-\varepsilon_i*\gamma}-a_iq^{-\varepsilon_j*\gamma}=0,\qquad
\forall \ 1\leq i<j\leq n,
$$
and its solution is
$a_j=a_iq^{-\varepsilon_j*\gamma+\varepsilon_i*\gamma}$, for
$1\le i<j\le n$, that is,
$$
a_j=a_1q^{-\varepsilon_j*\gamma}, \qquad \forall \ 1\le j\le n.
$$
So, $\dim\ker d^1|_{\Omega^{(1)}_\gamma}=1$, $\ker d^1|_{\Omega^{(1)}_\gamma}=\im d^0|_{\Omega^{(0)}_\gamma}$.
$\im d^0|_{\Omega^{(0)}_\gamma}=\ker d^1|_{\Omega^{(1)}_\gamma}$,
moreover, $\dim d^1(\Omega^{(1)}_\gamma)=\dim \Omega^{(1)}_\gamma-\dim\ker d^1|_{\Omega^{(1)}}$ $=\binom{n-k_\gamma-h_\gamma-1}{1}$.

When
$\Omega^{(0)}_\gamma=0$ but $\Omega^{(1)}_\gamma\ne0$, i.e., $k_\gamma=1$:  by Lemma 4.5, $\exists \, ! \ i$, such that
$\gamma_i=m\ell$ and $\dim \Omega^{(1)}_\gamma=1$. This implies $\gamma=m\ell\varepsilon_i$, and
$\Omega^{(1)}_\gamma=\span_k\{\,x^{((m\ell-1)\varepsilon_i)}\otimes
dx_i\,\}=\ker d^1|_{\Omega^{(1)}_\gamma}$.

\smallskip
Now for $s>1$, suppose for any $s'\le s$, the assertions are true. We
consider the case $s{+}1$:

Assume that
$$
d^{s+1}\Bigl(\sum_{i_1<\cdots<i_{s+1}}a_{i_1\cdots i_{s+1}}
x^{(\gamma-\sum^{s+1}_{j=1}\varepsilon_{i_j})}\otimes dx_{i_1}\wedge\cdots
\wedge dx_{i_{s+1}}\Bigr)=0,
$$
we can obtain a system of linear equations with indeterminates $a_{i_1{\cdots}i_{s+1}}$
$(1\le i_1<{\cdots}<i_{s+1}\le n)$,
$$
\sum^{s+2}_{j=1}a_{i_1\cdots \widehat{i_j}\cdots
i_{s+2}}(-1)^{j-1}q^{-\varepsilon_{i_j}*\gamma}=0,\qquad \forall \
1\le i_1<{\cdots}<i_{s+2}\le n.\leqno({\diamond})
$$

Set $P=\{\,i_1{\cdots}i_{s+1}\mid i_1<{\cdots}<i_{s+1},
x^{(\gamma{-}\sum_{j=1}^{s+1}\varepsilon_{i_j})}\otimes dx_{i_1}\wedge{\cdots}\wedge
\,dx_{i_{s+1}}\ne 0\}$, $Q=\{\,i_1{\cdots}i_{s+2}\mid
i_1<{\cdots}<i_{s+2}, x^{(\gamma{-}\sum_{j=1}^{s+2}\varepsilon_{i_j})}\otimes
dx_{i_1}\wedge{\cdots}\wedge dx_{i_{s+2}}\neq0\}$.
Denote $p=\#P$, $q=\#Q$.
Order lexicographically the words in $P$ and $Q$ respectively in column to get two column vectors $P$, $Q$.
Write $X=(a_{i_1\cdots i_{s+1}})_{i_1\cdots
i_{s+1}\in P}$. Thereby, we express the system $({\diamond})$ of $q$ linear equations with $p$ indeterminates $a_{i_1\cdots
i_{s+1}}$ as a matrix equation
$AX=0$, where the coefficients matrix $A$ is of size $q\times p$.

When $\Omega^{(s)}_\gamma\ne 0$,
i.e., $k_\gamma\le s$: there exists a unique longest word $\jmath_1\cdots\jmath_{k_{\gamma}}$ such that each $\gamma_{\jmath_r}=m\ell$. By definition, $\jmath_1\cdots\jmath_{k_{\gamma}}$ must be a subword of any word $i_1\cdots i_{s+1}$ in $P$,
and each $\gamma_{i_j}\ne 0$. Now set $b=min\{\,i\mid\gamma_i\neq0, \ 1\leq i\leq n\,\}$.
Owing to the lexicographic order adopted in $X$, it is easy to see that there is a diagonal submatrix $\text{\rm diag}\{q^{-\varepsilon_b*\gamma},\cdots,q^{-\varepsilon_b*\gamma}\}$ with order
$\binom{n-k_\gamma-h_\gamma-1}{s+1-k_\gamma}$ in the top right corner of $A$, which is provided by the front
$\binom{n-k_\gamma-h_\gamma-1}{s+1-k_\gamma}$ equations
corresponding to those words $i_1\cdots i_{s+2}$ with the beginning letter $i_1=b$. Thus, $\text{\rm rank}\,A\ge \binom{n-k_\gamma-h_\gamma-1}{s+1-k_\gamma}$, and
$\dim \ker d^{s+1}|_{{\Omega^{(s+1)}_\gamma}}=\dim\Omega^{(s+1)}_\gamma-\text{\rm rank}\,A \leq
\binom{n-k_\gamma-h_\gamma}{s+1-k_\gamma}-\binom{n-k_\gamma-h_\gamma-1}{s+1-k_\gamma}=\binom{n-k_\gamma-h_\gamma-1}{s-k_\gamma}$.

Note that $\im d^s\subseteq
\ker d^{s+1}$ and
$d^s\Omega^{(s)}_\gamma\subseteq \Omega^{(s+1)}_\gamma$. By the inductive
hypothesis,
$\dim\im d^{s}|_{{\Omega^{(s)}_\gamma}}=\binom{n-k_\gamma-h_\gamma-1}{s-k_\gamma}$,
so $\dim\ker d^{s+1}|_{{\Omega^{(s+1)}_\gamma}}\ge
\binom{n-k_\gamma-h_\gamma-1}{s-k_\gamma}$.

Therefore, we get
$\dim\ker d^{s+1}|_{{\Omega^{(s+1)}_\gamma}}=\binom{n-k_\gamma-h_\gamma-1}{s-k_\gamma}
$ $=\dim\im d^s|_{\Omega^{(s)}_\gamma}$, and
\begin{gather*}
\ker d^{s+1}|_{{\Omega^{(s+1)}_\gamma}}=\im d^s|_{\Omega^{(s)}_\gamma},\\
\im d^{s+1}|_{{\Omega^{(s+1)}_\gamma}}=\binom{n{-}k_\gamma{-}h_\gamma}{s{+}1{-}k_\gamma}{-}\binom{n{-}k_\gamma{-}h_\gamma{-}1}{s{-}k_\gamma}
=\binom{n{-}k_\gamma{-}h_\gamma{-}1}{s{+}1{-}k_\gamma}=\text{\rm rank}\,A.
\end{gather*}

When $\Omega^{(s)}_\gamma=0$ but
$\Omega^{(s+1)}_\gamma\ne 0$, i.e., $k_\gamma=s{+}1$:
there are $s{+}1$ $\gamma_i's$ equal to $m\ell$ and $\dim \Omega^{(s+1)}_\gamma=1$. In this
case, set $\gamma_{i_1}=\gamma_{i_2}=\cdots=\gamma_{i_{s+1}}=m\ell$.
Then $\gamma=\sum^{s+1}_{j=1}m\ell\varepsilon_{i_j}$ with $1\leq
i_1<\cdots<i_{s+1}\leq n$, and
$$
\Omega^{(s+1)}_\gamma=\span_k\bigl\{\,
x^{(\sum^{s+1}_{j=1}(m\ell-1)\varepsilon_{i_j})}\otimes
dx_{i_1}\wedge\cdots \wedge dx_{i_{s+1}}\,\bigr\}=\ker
d^{s+1}|_{\Omega^{(s+1)}_\gamma}.
$$

This completes the proof.
\end{proof}

\noindent
{\it 4.4. Cohomology modules.} We will concern the module structure on $H^s(\Omega_q(n,\bold m))$.
\begin{defi}
Let $V(\epsilon_1,\cdots,\epsilon_{n-1})$ be a one-dimensional
$\mathfrak{u}_q(\mathfrak{sl}_n)$-module. It is called a
sign-trivial module if for  $0\neq v\in V(\epsilon_1,\cdots,\epsilon_{n-1})$, $e_i.\, v=f_i.\, v=0$
and $K_i.\, v=\epsilon_i v$, where $\epsilon_i=\pm1$, for $i=1,\cdots,n{-}1$.
\end{defi}

\begin{theorem} For any $s$ $(0\leq s\leq n)$,
each cohomology group $H^s(\Omega_q(n,\bold m))$ is isomorphic to the direct sum of $\binom{n}{s}$ $($sign-$)$trivial
$\mathfrak{u}_q(\mathfrak{sl}_n)$-modules when $q$ is an $\ell$-th $($resp.  $2\ell$-th but $m$ is odd\,$)$ root of unity or
$m$ is even.
\end{theorem}
\begin{proof}
When $s=0$, the statement is clear.

It suffices to consider the cases when $1\leq s\leq n$.
By Theorem 4.6, we have
$$
H^s(\Omega_q(n,\bold m))
\cong\span_k\bigl\{
x^{(\sum^s_{j=1}(m\ell-1)\varepsilon_{i_j})}\otimes dx_{i_1}\wedge
\cdots\wedge dx_{i_s}\,\big|\,1\leq i_1<\cdots<i_s\leq
n\bigr\}.
$$

Denote by $[x^{(\sum^s_{j=1}(m\ell-1)\varepsilon_{i_j})}\otimes
dx_{i_1}\wedge {\cdots}\,\wedge dx_{i_s}]$, the image of
$x^{(\sum^s_{j=1}(m\ell-1)\varepsilon_{i_j})}\otimes dx_{i_1}\wedge
{\cdots}\,\wedge dx_{i_s}\in\ker d^s|_{\Omega_\gamma^{(s)}}$ in $H^s(\Omega_q(n,\bold m))$, where
$\gamma=\sum_{j=1}^s(m\ell)\varepsilon_{i_j}$.

Consider the actions of the generators $e_i,\, f_i,\, K_i,\,
K^{-1}_i$ $(1\le i\le n{-}1)$ of
$\mathfrak{u}_q(\mathfrak{sl}_n)$ on
$[\,x^{(\sum^s_{j=1}(m\ell-1)\varepsilon_{i_j})}\otimes dx_{i_1}\wedge
\cdots\wedge dx_{i_s}]$.

\medskip
(1) For any $e_h$: if
$e_h.\,(x^{(\sum^s_{j=1}(m\ell-1)\varepsilon_{i_j})}\otimes
dx_{i_1}\wedge {\cdots}\wedge dx_{i_s})\neq 0$,
then $h{+}1\in
\{\,i_1, \dots, i_s\}$ and $h\in\{\,i_1{-}1, \cdots, i_s{-}1\}-\{\,i_1, \cdots, i_s\}$. Write
$\gamma'=\sum^s_{j=1}(m\ell)\varepsilon_{i_j}{+}\varepsilon_h$ ${-}\varepsilon_{h+1}$, then
$e_h.\,(x^{(\sum^s_{j=1}(m\ell-1)\varepsilon_{i_j})}\otimes
dx_{i_1}\wedge \cdots\wedge dx_{i_s})\in \Omega^{(s)}_{\gamma'}\cap\ker d^s=\ker d^s|_{\Omega^{(s)}_{\gamma'}}$.

Since
$k_{\gamma'}=s{-}1$, $h_{\gamma'}=h_\gamma{-}1=n{-}s{-}1$, by Lemma 4.5,
$\dim\Omega^{(s-1)}_{\gamma'}=1$. So now the problem reduces to Case (i) in the proof of Theorem 4.6.
We then obtain
$\ker d^s|_{\Omega^{(s)}_{\gamma'}}=\im d^{s-1}|_{\Omega^{(s-1)}_{\gamma'}}\ne 0$.
Thus
$e_h.\,(x^{(\sum^s_{j=1}(m\ell-1)\varepsilon_{i_j})}\otimes
dx_{i_1}\wedge \cdots\wedge dx_{i_s})\in\im d^{s-1}$, namely,
$e_h.\,[\,x^{(\sum^s_{j=1}(m\ell-1)\varepsilon_{i_j})}\otimes
dx_{i_1}\wedge \cdots\wedge dx_{i_s}]=0$.

\smallskip
Therefore,  $e_h$ acts trivially on $H^s(\Omega_q(n,\bold m))$.

\medskip
(2) Dually, we can check that $f_h$ trivially acts on $H^s(\Omega_q(n,\bold m))$.
\medskip

(3) For $K_i^{\pm1}$: we have
\begin{equation*}
\begin{split}
K_i^{\pm1}&.\,[\,x^{(\sum^s_{j=1}(m\ell-1)\varepsilon_{i_j})}\otimes
dx_{i_1}\wedge \cdots\wedge
dx_{i_s}]\\
\quad&=q^{\pm(\gamma_i-\gamma_{i+1})}[\,x^{(\sum^s_{j=1}(m\ell-1)\varepsilon_{i_j})}\otimes
dx_{i_1}\wedge \cdots\wedge dx_{i_s}].
\end{split}
\end{equation*}

Notice that $\gamma_i-\gamma_{i+1}=\pm m\ell$ or $0$, for
$\gamma=\sum^s_{j=1}(m\ell)\varepsilon_{i_j}$, where
$r_i\in\mathbb{Z}$.

\smallskip
(i) When $q$ is the $\ell$-th primitive root of unity or $m$ is even,
$q^{\pm(\gamma_i-\gamma_{i+1})}=1$, the submodule generated by
$[\,x^{(\sum^s_{j=1}(m\ell-1)\varepsilon_{i_j})}\otimes dx_{i_1}\wedge
\cdots\wedge dx_{i_s}]$ is a trivial module.

\smallskip
(ii) When $q$ is the $2\ell$-th primitive root of unity but $m$ odd, $q^{\pm(\gamma_i-\gamma_{i+1})}=\pm1$, then the submodule
generated by $[\,x^{(\sum^s_{j=1}(m\ell-1)\varepsilon_{i_j})}\otimes
dx_{i_1}\wedge \cdots\wedge dx_{i_s}]$ is a sign-trivial module.

\smallskip
Hence, $H^s(\Omega_q(n,\bold m))$ is
isomorphic to the direct sum of $\binom{n}{s}$ sign-trivial
$\mathfrak{u}_q(\mathfrak{sl}_n)$-modules.
\end{proof}

\noindent
{\it 4.5. Quantum de Rham cohomologies $H^s(\Omega_q(n))$.} In this final subsection, we turn to give
a description of the cohomologies for the quantum de Rham complex $(\Omega_q(n), d^\bullet)$.
Actually, Lemma 4.5 and the result of Case (i) in the proof of Theorem 4.6 are still available to the
$(\Omega_q(n), d^\bullet)$.
\begin{prop}
For the quantum de Rham complex $(\Omega_q(n), d^\bullet)$ over $\mathcal{A}_q(n)$$:$
$$
0 {\longrightarrow}\Omega_q(n)^{(0)} \stackrel{d^0}
{\longrightarrow}
\cdots\stackrel{d^{s-1}}
{\longrightarrow}\Omega_q(n)^{(s)}\stackrel{d^s}
{\longrightarrow}\Omega_q(n)^{(s+1)}\stackrel{d^{s+1}}
{\longrightarrow}\cdots
\stackrel{d^{n-1}}
{\longrightarrow}\Omega_q(n)^{(n)}\stackrel{d^{n}}
{\longrightarrow}0,
$$
one has $H^s(\Omega_q(n))=\delta_{0,s}k$, for any $s=0, 1, \cdots, n$.
\end{prop}
\begin{proof} Clearly, we have $H^0(\Omega_q(n))=k$.

For any given $\gamma\in\mathbb Z_+^n$, since each $d^s$ preserves the weight-gradings, we have
$$
0 {\longrightarrow}\Omega_q(n)^{(0)}_\gamma \stackrel{d^0}
{\longrightarrow}
\cdots\stackrel{d^{s-1}}
{\longrightarrow}\Omega_q(n)^{(s)}_\gamma\stackrel{d^s}
{\longrightarrow}\Omega_q(n)^{(s+1)}_\gamma\stackrel{d^{s+1}}
{\longrightarrow}\cdots
\stackrel{d^{n-1}}
{\longrightarrow}\Omega_q(n)^{(n)}_\gamma\stackrel{d^{n}}
{\longrightarrow}0.
$$
By definition, $H^s(\Omega_q(n))=\bigoplus_{\gamma\in\mathbb Z_+^n}\ker d^s|_{\Omega_q(n)^{(s)}_\gamma}/\im d^{s-1}|_{\Omega_q(n)^{(s-1)}_\gamma}$. So, for the given $\bold 0<\gamma\in \mathbb Z_+^n$, there exists
an $m\in\mathbb N$, such that $m\ell>|\,\gamma\,|$. This means that
$\Omega_q(n)_\gamma^{(s)}=\Omega_q(n,\bold m)^{(s)}_\gamma$, for any $s\ge 1$, and $k_\gamma=0$.
So, Lemma 4.5 is adapted to our case, namely,
$\Omega_q(n)^{(s-1)}_\gamma=\Omega_q(n,\bold m)^{(s-1)}_\gamma\neq 0$, for $1\leq s\leq n$. According to the proof of Theorem 4.6, the result of Case (i) works, that is,
$\ker\,d^s|_{\Omega_q(n)^{(s)}_\gamma}=\im\, d^{s-1}|_{\Omega_q(n)^{(s-1)}_\gamma}$, for any given $\gamma$.
This implies $H^s(\Omega_q(n))=0$ for $s\ge 1$.
\end{proof}

\bibliographystyle{amsalpha}

%\end{CJK*}

\end{document}